\documentclass{article}
\usepackage[utf8]{inputenc}
\usepackage{amssymb}
\usepackage[T1]{fontenc}
\usepackage{amsmath}
\usepackage{amsthm}
\usepackage{latexsym}
\usepackage{xcolor}
\usepackage{graphicx}

\usepackage{subcaption}

\usepackage{enumerate}
\usepackage{hyperref} 
\usepackage{stmaryrd}

\newtheorem{defi}{Definition}
\newtheorem{remark}{Remark}
\newtheorem{lemma}{Lemma}
\newtheorem{theorem}[lemma]{Theorem}

\newtheorem{proposition}[lemma]{Proposition}
\def\R{\mathbb{R}}
\def\N{\mathbb{N}}

\title{Approximation and stability results for the parabolic FitzHugh-Nagumo system with
combined rapidly oscillating sources\footnotemark[1]}

\author{Eduardo Cerpa\footnotemark[3] , \,Mat\'ias Courdurier\footnotemark[4] , \,Esteban Hern\'andez\footnotemark[5] ,\\
\,Leonel E. Medina\footnotemark[6], \,Esteban Paduro\footnotemark[2]\,\,\footnotemark[3]}
\date{April 2025}
\begin{document}
\maketitle

\footnotetext[1]{This work has been partially supported by ANID Millennium Science Initiative
Program through Millennium Nucleus for Applied Control and Inverse Problems NCN19-161, Fondecyt 11190822, and a grant from the Simons Foundation during a visit to the Programme RNT at the Isaac Newton Institute, Cambridge, supported by EPSRC grant no EP/R014604/1.
}
\footnotetext[2]{Corresponding author. E-mail: {\tt esteban.paduro@mat.uc.cl}}

\footnotetext[3]{Instituto de Ingenier\'ia Matem\'atica y Computacional, Pontificia Universidad Cat\'olica de Chile,  
Santiago, Chile.
}
\footnotetext[4]{Departamento de Matem\'atica, Facultad de Matem\'aticas, Pontificia Universidad Cat\'olica de Chile. 
Santiago, Chile.
}
\footnotetext[5]{Departamento de Matem\'atica, Universidad T\'ecnica Federico Santa Mar\'ia, 
Valpara\'iso, Chile.
}
\footnotetext[6]{Departamento de Ingenier\'ia Inform\'atica, Universidad de Santiago de Chile, 
Santiago, Chile. 
}

\begin{abstract}

The use of high-frequency currents in neurostimulation has received increased attention in recent years due to its varied effects on tissues and cells. Nonlinear differential equations are commonly used as models for Neurons, and averaging methods are suitable for addressing questions like stability when considering single-frequency sources. A recent strategy called temporal interference stimulation uses electrodes to deliver sinusoidal signals of slightly different frequencies. Thus, classical averaging cannot be directly applied. 
This paper considers the one-dimensional FitzHugh-Nagumo system under the effects of a source composed of two sinusoidal terms in time and decaying in space. We develop a new averaging strategy to show that the solution of the system can be approximated by an explicit highly-oscillatory term plus the solution of a simpler, non-autonomous system. One of the main novelties is an extension of the contracting rectangles method to the case of parabolic equations with space and time-depending coefficients.

\end{abstract}

\vspace{0.2 cm}


\section{Introduction}

\subsection{Motivation}\label{subsec_motivation}
The cable equation is a type of nonlinear parabolic system that is used in neuroscience to study the dynamics of the membrane voltage in neurons and axons \cite{HH1952,hastings_mathematical_1975,mcneal_analysis_1976,rattay_analysis_1986,warman_modeling_1992}. In this context, the FitzHugh-Nagumo (FHN) system \cite{FITZHUGH1961445} describes membrane voltage dynamics, and has been used to understand neuron responses to electrical stimulation. Electrical neurostimulation is a technique used to treat the symptoms of various nervous system disorders and diseases, including Parkinson’s disease, epilepsy, and chronic pain, among others \cite{Krames2009}. To improve the success of these therapies, a key step is to advance the understanding of how action potentials, \textit{i.e.}, abrupt changes in membrane potential that can propagate along neurons, are generated upon the application of electric fields from external current sources. In mathematical terms, the biological phenomenon of action potential generation and propagation can be linked to traveling wave solutions of a nonlinear partial differential equation (PDE) posed on the whole line as the following FHN system  
\begin{equation}\label{FHN_pde_non_centered_intro}
\left\{
\begin{array}{rcll}
\partial_t f - \partial_x^2 f  &=& f - f^3/3  - g + I(x,t) &,\quad \forall x\in\R,t\geq 0,\\
\partial_t g  &=& \varepsilon(f - \gamma g + \beta) &,\quad \forall x\in\R,t\geq 0,\\
f(x,0) = f_0(x), && g(x,0) = g_0(x) &,\quad \forall x\in\R,
\end{array}\right.
\end{equation} where $f=f(x,t)$ represents the membrane voltage, $g=g(x,t)$ is a recovery variable, $I(x,t)$ is the input current, $f_0=f_0(x),g_0=g_0(x)$ are the initial data, and $\varepsilon, \gamma,\beta$ are positive constants. Classical results on the well-posedness of \eqref{FHN_pde_non_centered_intro} can be found on \cite{MR425397,rauch-smoller,MR511479}. Understanding and characterizing traveling waves for \eqref{FHN_pde_non_centered_intro} is crucial, but because of their nonlinear nature,  they are difficult to study by linearization approaches. Prior work has studied the existence \cite{hastings_single_1982}, stability \cite{jones_stability_1984,bressloff_traveling_2000}, and numerical approximation \cite{namjoo_numerical_2018,al-juaifri_finite_2023} of traveling waves in FHN equations. For similar nonlinear parabolic equations, there are different profiles for the traveling wave solutions, including front-like in Fisher-KPP type systems \cite{fisher_wave_1937,kolomgorov1937study} and pulse-like traveling waves \cite{carpenter_geometric_1977,conley_traveling_1975,langer_existence_1980,hastings_existence_1976}, for single equations or systems. Particularly interesting are some recent results for propagation in periodic media \cite{liang,ding,hamel,hamel2,nadin,karsten3}.

Different authors have considered the problem of stimulation for the FHN model from an applied viewpoint \cite{taghipour_farshi_assessment_2016,chen_multi_scale_2021,ambrosio_periodically_2022} and in recent years, the use of rapidly oscillating current sources, in the kilohertz-frequency range, has received increasing attention in neurostimulation.  This is because incorporating high-frequency components in the source may facilitate reaching deeper structures from the surface \cite{Medina2014} and can block the conduction of action potentials in axons \cite{Bhadra2007}. This brings new difficulties associated with the key role of high-frequency sources, affecting even the efficiency of numerical simulations. 

In this paper, we consider the one-dimensional FHN  system \eqref{FHN_pde_non_centered_intro} subjected to the so-called temporal interference stimulation \cite{Mirzakhalili2020,Gomez-Tames2021}, which is composed of two sources that are sinusoidal in time and quadratically decaying in space. In practice, these kinds of sources are used to look for new neuron activation profiles but bring new technical difficulties from a mathematical viewpoint. To study this setting, we develop a new averaging strategy to prove that
when 1) the frequencies of the interferential currents are sufficiently high and 2) the amplitude parameters are not too large, the full system can be approximated by 
an explicit highly-oscillatory term plus the solution of a simpler -albeit non-autonomous- system, for which we present some stability results. In a previous work \cite{cerpa_partially_2022}, we studied an ordinary differential equation FHN model, but significant challenges appeared when trying to extend the result to the PDE system. Our approach is inspired by the averaging method (see, e.g., \cite{khalil2013}) used in prior work to study oscillating source terms in the form of one sinusoidal signal \cite{Ratas2012,Weinberg2013}. With only one rapidly oscillatory sinusoidal term in the source, the approximated averaged system is autonomous, allowing for direct stability analysis. Here, we study the case of two rapidly oscillatory sinusoidal terms, obtain an approximated averaged system that is non-autonomous, and tackle the difficulties associated with stability analyses of such a system. In addition to the averaging method, another key component of this work is the technique of contracting rectangles, which has been used for the FHN model \cite{rauch-smoller} and other related problems \cite{conway1977, conway_large_1978,valencia_invariant_1990,ZHOU202239}.

\subsection{Averaging of PDE with time oscillatory terms}
One of the main technical tools used to study systems involving high-frequency source terms is averaging. For PDE, we can find the works \cite{karsten,matthiesExponentialAveragingRapid2008,karsten2} and a very complete review in \cite[Appendix E]{sanders}. Let us present our problem in the context of classic averaging and describe the obstacles to a direct application of that method, hence requiring a modified approach. Consider problem \eqref{FHN_pde_non_centered_intro} in a general abstract form as a system for $u=(f,g)$
\begin{equation}\label{averaging_problem}
\partial_t u + \mathfrak{L} u= F(u,t,\omega,x),\quad x\in \R,\quad t\geq0,
\end{equation}
where $\mathfrak{L}$ is a linear (possibly singular) elliptic (acting in the space variable) operator, $\omega$ represents the high-frequency of a periodic term, and $F$ is a nonlinear function that might include non-homogeneous source terms. After a change of variables $u = v + h(t)$, $\omega t = s$ (for an appropriate $h$), the system \eqref{FHN_pde_non_centered_intro} can be recast as
\begin{equation}\label{averaging_problem2}
\partial_s v + \frac{1}{\omega} \mathfrak{L} v =\frac{1}{\omega} \tilde{F}(v,s,\omega,x),
\end{equation}
where $\frac{1}{\omega}$ is a small parameter because we consider a high frequency $\omega$. At this point, there are a few approaches that we could try to apply, such as classical averaging in some appropriate Banach space. However, our $\tilde{F}$ is not periodic over $t$, and we do not have a structural short time interval to average over. In \cite{hale_averaging_1990}, averaging over an infinite time interval is shown to work for almost periodic functions of the form $F(u,t,\omega,x) = G(u, t \omega)$. But there are two obstacles when applying this directly to \eqref{averaging_problem2}. The first one is the presence in $\tilde{F}$ of a non-autonomous term in $x$, although this can likely be overcome using a semi-group approach \cite{hale_averaging_1990}. The second issue is more fundamental: in $\tilde{F}$, the dependence on $(s,\omega)$ cannot be written as a dependence only on $s\omega$. Indeed, a key element in temporal interference stimulation is the effect of the slowly varying envelope arising from the arithmetic addition of the two highly oscillatory sinusoidal signals. Averaging in the quasiperiodic setting is considered in \cite{simoAveragingFastQuasiperiodic1994} for ODEs and \cite{matthiesExponentialAveragingRapid2008} for PDEs. When applied to our problem, the Diophantine condition required for quasiperiodic averaging translates into an additional hypothesis of the ratio of the frequencies of the sources being irrational, and the result leads to an autonomous approximation of the problem. Partial averaging for ODEs has also been considered under the additional hypothesis that the averaged system has an exponentially stable equilibrium point \cite{peuteman_exponential_2002,grammel_exponential_2003,hale_ordinary_2009}. Still, we cannot use such results since our averaged system has no static equilibrium.

An additional challenge arises when establishing the time length for which  \eqref{averaging_problem2} can be approximated by its averaged version. Classical averaging results of ODE suggest that \eqref{averaging_problem2} can be compared with the solution of the corresponding averaged system with an error of $O(1/\omega)$ for a time length  $O(\omega)$. However, for the original variables $t=s/\omega$, such an estimate only works for a time length $O(1)$. For ODEs, the estimates can be extended for all time lengths by requiring the additional condition of exponential stability around the equilibrium point for the averaged system \cite{hale_ordinary_2009,khalil2013}. However, the averaged version of \eqref{averaging_problem2} is a non-autonomous PDE, hence additional considerations are needed to establish approximation results for longer times. To overcome these difficulties, we propose an averaging method over the fast oscillations that retains the slower oscillating envelope formed by the two sinusoidal signals, similar to our approach for the ODE system \cite{cerpa_partially_2022}. Averaging over the fast oscillations results in a non-autonomous approximated system, which adds many technical challenges compared with classical averaging but allows for estimates that improve as the frequency $\omega$ increases. This particular averaging strategy, which goes through the analysis of a related non-autonomous equation, can also be used to obtain a separation of timescales among the linear high-frequency response and the nonlinear phenomena arising in some modulated high-frequency excitation of the FHN ODE system (see \cite{cerpa2024impact}).

Since we look for estimates for all times, uniform exponential stability around an equilibrium point for the averaged system would be useful. Unfortunately, such a result is not available for the PDE system.  Instead, we use the following approach to extend the estimates in time: first, via a fixed point argument we show that up to linear order terms, the approximation error is $O(\frac{1}{\omega})$, this is important since the right-hand side of \eqref{averaging_problem} depends explicitly on both $x$ and $t$, which makes the analysis harder; second, using the contracting rectangles method \cite{rauch-smoller} we can study the nonlinear part of the approximation error.

\subsection{Setting}
We consider a model axon under the effect of a current generated by two point sources with different frequencies. The axon is represented by an infinite straight line parametrized by $x\in\R$. The first point source has coordinate $x=0$ and is located at a distance $d_1$ from the axon, while the second source has coordinate $x=x_0$ and is at a distance $d_2$ from the axon. Thus, the effect of point sources, with amplitudes $\omega_1a, \omega_2b\in\R$, over the axon can be written as a source term
\begin{equation}\label{input_current}
I(x,t) = A(x)\omega_1 \cos(\omega_1 t) +  B(x)\omega_2\cos(\omega_2 t),
\end{equation}
with $A(x)$ and $B(x)$ that decay in space. To simplify the presentation, we will assume quadratic decay along the direction of the axon (other decays can also be tackled with the techniques presented here)
\begin{equation}\label{decay_AxBx}
A(x) = \frac{a}{d_1^2+x^2},\quad B(x) = \frac{b}{d_2^2+(x-x_0)^2}.
\end{equation}
Here, $a$, $b \in \R$ does not depend on the frequency. We describe the sources with amplitudes scaled by frequency because we will see this is natural scaling (also observed in \cite{Ratas2012}). It is important to underline that since $\omega_1,\omega_2$ will be arbitrarily large, the source $I(x,t)$ can have a very large magnitude even if $a,b,$ are small. 

Regarding the different frequencies $\omega_1$ and $\omega_2$ in \eqref{input_current}, we note that we focused on the case of temporal interference currents defined by 
\begin{equation}\label{interferential}
   \omega_1 \gg 1,\quad\omega_2=\omega_1+\eta\,\, \quad\text{for some beat frequency }\, \eta>0.
\end{equation}

In practice, typical values for the frequency of the interferential sources are in the few kilohertz range, with a beat frequency in the order of tens of hertz.  We aim to study the FHN system for source terms like \eqref{input_current}, and we consider a slightly more general version of \eqref{FHN_pde_non_centered_intro} given for $\rho\geq0$ by 
\begin{equation}\label{FHN_pde_non_centered}
\left\{
\begin{array}{rll}
\partial_t f - \partial_x^2 f  &= f - f^3/3  - g + I(x,t) &,\quad \forall x\in\R,t\geq 0,\\
\partial_t g -\rho \partial_x^2 g &= \varepsilon(f - \gamma g + \beta) &,\quad \forall x\in\R,t\geq 0,\\
f(x,0) = f_0(x), & g(x,0) = g_0(x)&,\quad \forall x\in\R,
\end{array}\right.
\end{equation}
where $f=f(x,t)$ represents the voltage, $g=g(x,t)$ is a recovery variable, $f_0=f_0(x),\,g_0=g_0(x)$ are the initial data, and $\varepsilon$, $\gamma$, $\beta$ are positive constants. It is worth mentioning that this system is nonlinear, potentially allowing solutions to behave as traveling waves. From the biological point of view, the most relevant case is $\rho=0$, \textit{i.e.}, the classical FitzHugh-Nagumo system, but the case $\rho>0$ can be of interest if some spatial diffusion of the recovery variable needs to be included. Hence, we study the cases $\rho=0$ and $\rho>0$ in a unified manner (similarly to what was done in \cite{rauch-smoller}).

For the initial data of system \eqref{FHN_pde_non_centered}, we impose that
$$f_0(x)\to v_0,\quad g_0(x) \to w_0,\quad \text{ as } x \to \pm \infty,$$ where $(v_0,w_0)\in\R^2$ is the unique solution (see Definition \ref{defi_admissible_parameters} below) of 
\begin{equation}\label{defi_v0w0}
\left\{\begin{array}{rl}
0 &= v_0-v_0^3/3-w_0,\\
0 &= v_0-\gamma w_0 + \beta,\\
\end{array}\right.
\end{equation}
and consequently $(v_0,w_0)$ is a stationary solution for \eqref{FHN_pde_non_centered} when $I(x,t)=0$. We consider the system \eqref{FHN_pde_non_centered} centered at the stationary solution given by \eqref{defi_v0w0}. Hence, we set
$v = f- v_0$, $w = g- w_0$, $\bar{f}_0 = f_0 - v_0$ and $\bar{g}_0 = g(x) -w_0$ to get
\begin{equation}\label{FHN_pde}
\left\{\hspace{-1mm}
\begin{array}{rll}
\partial_t v - \partial_x^2 v &= (1- v_0^2)v  -v_0 v^2 - \frac{1}{3}v^3- w  + I(x,t), &\forall x\in\R,t\geq 0,\\
\partial_t w-\rho \partial_{x}^{2}w&= \varepsilon (v - \gamma w),  &\forall x\in\R,t\geq 0,\\
v(x,0) = \bar{f}_0(x),& w(x,0) = \bar{g}_0(x), &\forall x\in\R,
\end{array}
\right.
\end{equation}
with $\bar{f}(x)\to 0,\, \bar{g}(x)\to 0$ when $x\to\pm\infty$.
Regarding the parameters, the following definition establishes the possible choices.

\begin{defi}[Admissible set of parameters]\label{defi_admissible_parameters}
The parameters $\beta, \gamma >0$ are said to be admissible if
\begin{multline}\label{condition_H1}
\frac{(1-\gamma)^3}{\gamma^3} +\frac 9 4 \frac{\beta^2}{\gamma^2} > 0 \text{ and if }
\exists\delta,\, 0<\delta < 1/4 \text{ such that } \left(1+\frac{1}{\delta\gamma}\right)^{1/2}\left( 2-\frac{3+1/\delta}{\gamma}\right) + 3\frac{\beta}{\gamma}> 0.
\end{multline}
\end{defi}

\begin{remark}
This condition \eqref{condition_H1} is obtained after requiring the system \eqref{defi_v0w0} to have a unique solution and other technical conditions needed throughout the paper. More detailed explanations are given in Lemma \ref{properties_admissible_parameters}, where we show that the set of parameters satisfying \eqref{condition_H1} is not empty.
\end{remark}

\subsection{Main results}
The main purpose of this work is to present a novel tool for the analysis of nonlinear systems such as \eqref{FHN_pde} when considering highly oscillating source terms as those given by \eqref{input_current} and \eqref{interferential}. A key element in our analysis is an approximated system that we call the Partially Averaged System, which is defined as follows.

\begin{defi}[Partially Averaged System] Given $\beta$, $\gamma$, satisfying \eqref{condition_H1}, $\rho\geq0$, and $\varepsilon>0$ the Partially Averaged System of \eqref{FHN_pde} for a source term $I(x,t)$ given by \eqref{input_current} and \eqref{interferential}, is defined as 
\begin{equation}\label{FHN_pde_pas}
\left\{
\begin{array}{rl}
\partial_t V-\partial_{x}^2V  &= \left(1-v_0^2-\frac{A(x)^2}{2} - \frac{B(x)^2}{2} -A(x) B(x) \cos(\eta t) \right)V\\
&\hspace{0.2cm}-v_0V^2- \frac{1}{3}V^3 - W\\
&\hspace{0.2cm}-\left(\frac{A(x)^2}{2}+\frac{B(x)^2}{2}+A(x)B(x)\cos(\eta t)\right)v_0, \hfill\forall x\in\R, t\geq 0,\\
\partial_t W- \rho\partial_{x}^2 W &= \varepsilon\left( V- \gamma W \right), \hfill\forall x\in\R,t\geq 0,\\
V(x,0)  &= \bar{f_0}(x),\;  W(x,0) =\bar{g_0}(x), \hfill \forall x\in\R.
\end{array}
\right.
\end{equation}
The formal derivation of this system as an approximation is included in Appendix \ref{appendix_derivation}. 
\end{defi}

To study system \eqref{FHN_pde_pas} we need to extend the well-posedness results from \cite{rauch-smoller} to the non-autonomous case. We consider the following family of nonlinearities.

\begin{defi}\label{defi_space_vB}
Let us denote by $BC^0(\R)$ the space of real-valued functions defined on $\R$, which are bounded and uniformly continuous. Given a Banach space $\mathcal{B}\subset BC^0(\R)$ we consider the space $V(\mathcal{B})$ of vector-valued functions
$$F(U,x,t): \R^2\times \R \times [0,\infty) \to \R^2 $$
which can be written in the form
$$F(U,x,t) = f_{(0,0)}(x,t) + \sum_{1\leq |\alpha| \leq n} f_\alpha(x,t) U^\alpha,$$
where $n \in \N$, $\alpha = (\alpha_1, \alpha_2)$ is a multi-index, $|\alpha| = \alpha_1+ \alpha_2$, $U^\alpha = U_1^{\alpha_1} U_2^{\alpha_2}$, and the vector-valued functions $f_\alpha$ satisfy
\begin{itemize}
    \item $t\mapsto f_{(0,0)}(\cdot,t) \in L^{\infty}([0,\infty);\mathcal{B}\times \mathcal{B})$,
    \item for $1\leq |\alpha| \leq n$ we have
$t\mapsto f_\alpha(\cdot,t)\in L^{\infty}([0,\infty);\mathcal{B}\times \mathcal{B})$ or $t\mapsto f_\alpha(\cdot,t) = f_\alpha(t) \in L^{\infty}([0,\infty);\R^2)$.
\end{itemize}
\end{defi}

The local existence of solutions of the FHN system in the non-autonomous case is the most challenging aspect of such an extension. The proof of Theorem \ref{thm_well_posedness} is given in Section \ref{section_well_posedness}.

\begin{theorem}[Well posedness for non-autonomous systems] \label{thm_well_posedness} Let $\mathcal{B}$ denote either the Sobolev space $\mathcal{B}=W^{k,p}(\R)$ with $k,p\geq 1$, $kp >1$ or the space $\mathcal{B}=BC^0(\R)\cap L^p(\R)$ with $p\geq 1$ and let $F(U,x,t)\in V(\mathcal{B})$.
\begin{itemize}
\item[(i)] Then, for any $U_0 \in {\mathcal{B}\times \mathcal{B}}$, there exists a constant $t^*>0$, depending only on $F$ and $\|U_0\|_\infty$ such that the initial value problem 
\begin{equation}\label{eqn_generalized_parabolic}
\partial_t U - A \partial_{x}^2 U = F(U,x,t), \quad x\in \R,\quad t\geq t_0,
\end{equation}
where $A =  diag(1,\rho)$, with $\rho \geq 0$ and data $U(t_0) = U_0$ has a unique solution in $C([t_0,t_0+t^*];{\mathcal{B}\times \mathcal{B}})$.

\item[(ii)] Additionally, assume the functions in $\mathcal{B}$ are continuous and decay to 0 at $\pm \infty$, assume there exists a rectangle $R = [-L,L]\times[-S,S]\subset \R^2$ such that for each $\vec{v} \in \partial  R$ and every outward pointing unit vector $n(\vec{v})$ we have
\begin{equation*}
\sup_{x\in\R}n(\vec{v}) \cdot F(\vec{v},x,t) < 0, \quad t>0,
\end{equation*}
and assume that for some $\hat{\epsilon} >0$ 
$$
(1+\hat{\epsilon})U_0(x) \in R,\quad \forall x\in \R,$$
then, there is a unique solution $U\in C([0,\infty);{\mathcal{B}\times \mathcal{B}})$ of \eqref{eqn_generalized_parabolic} with $U(0) = U_0$ which also satisfies
$$U(x,t) \in R, \quad \forall x\in \R, t>0.$$
\end{itemize}
\end{theorem}

Our second result establishes that for small temporal interference stimulation, the Partially Averaged System satisfies some uniform estimates in time and space, providing stability. 

\begin{theorem}[Stability for the Partially Averaged System]\label{thm_estimates_pas} 
Let $\mathcal{B}$ denote either the space $\mathcal{B}=W^{k,p}(\R)$ with $k,p\geq 1$, $kp >1$  or the space $\mathcal{B}=BC^0(\R)\cap L^p(\R)$ with $p\geq 1$.  Let $\varepsilon >0$ and let $\beta$, $\gamma$ be such that \eqref{condition_H1} holds. Let $(v_0,w_0)\in \R^2$ be the unique solution of \eqref{defi_v0w0}. Then, given any open neighborhood $\mathcal{O}$ of $(0,0)$ there exist $N >0$ and a rectangle $R = [-L,L]\times [-S,S]\subset \mathcal{O}$ such that if  $A(x)$, $B(x)$ defined in \eqref{decay_AxBx} satisfy
$$\frac{|a|}{d_1^2}+\frac{|b|}{d_2^2} \leq N$$ 
and the initial data $(f_0(x)-v_0,g_0-w_0)\in \mathcal{B}\times \mathcal{B}$ satisfies for some $\hat{\epsilon}\in(0,1)$, 
$$(1+\hat{\epsilon})(f_0(x)-v_0,g_0(x)-w_0)\in R,$$
then, the initial value problem \eqref{FHN_pde_pas} with initial data $(f_0(x)-v_0,g_0(x)-w_0)$ has a unique solution  $(V,W)\in C([0,\infty);\mathcal{B}\times \mathcal{B})$ 
 satisfying 
$$(V(x,t),W(x,t))\in R, \quad  \forall x\in\R, \quad t \geq 0.$$
\end{theorem}

Our last result establishes a precise notion of how the Partially Averaged System \eqref{FHN_pde_pas} can be used to approximate the solution of system \eqref{FHN_pde_non_centered}, which is the one with biological interest but hard to work with. 

\begin{theorem}[Existence for original system and approximation result]\label{thm_approximation_pas}
Let $\mathcal{B}$ denote the space $\mathcal{B}=W^{k,p}(\R)$ with $k\geq 3$, $p\geq 1$.  Let $\varepsilon >0$ and $\beta$, $\gamma$ be such that \eqref{condition_H1} holds. Let $(v_0,w_0)\in \R^2$ be the unique solution of \eqref{defi_v0w0}. Given any open neighborhood $\mathcal{O}$ of $(0,0)$ and any $\mu >0$, there exist $N>0$ and a rectangle $R= [-L,L],[-S,S]\subset \mathcal{O}$ such that if $A(x)$ and $B(x)$ defined in \eqref{decay_AxBx} satisfy
$$\frac{|a|}{d_1^2}+\frac{|b|}{d_2^2} \leq N$$ 
and the initial data $(f_0(x)-v_0,g_0(x)-w_0)\in \mathcal{B}\times \mathcal{B}$ satisfies for some $\hat{\epsilon}\in(0,1)$,
$$(1+\hat{\epsilon})(f_0(x)-v_0,g_0(x)-w_0)\in R,$$
then, for all $\omega_1$ large enough the solution $(f,g)$ of system \eqref{FHN_pde_non_centered} with initial data $(f_0, g_0)$ exists in $C([0,\infty),\mathcal{B}\times \mathcal{B})$ and can be approximated using the solution $(V,W)$ of the Partially Averaged System \eqref{FHN_pde_pas} with initial data $(f_0-v_0, g_0-w_0)$ in the following way
\begin{align*}
|f -(v_0 + V+ A(x) \sin(\omega_1 t) + B(x) \sin(\omega_2 t) )| &\leq \mu,\quad\forall x\in \R, \quad t>0,\\
|g - (w_0 + W)| &\leq \mu,\quad\forall x\in \R, \quad t>0.
\end{align*}
\end{theorem}

\begin{remark}
    Theorem \ref{thm_approximation_pas} requires $A(x)$, $B(x)$ to be small, but since $\omega_1$ can be arbitrarily large, the source $I(x,t)$ (given by equation \eqref{input_current}) can have a very large magnitude. I.e., Theorem \ref{thm_approximation_pas} does not reduce to the existence of solutions for the FHN system under small perturbations.
\end{remark}

\subsection{Organization of the paper}
In Section \ref{section_well_posedness}, we prove Theorem \ref{thm_well_posedness} on the well-posedness of some non-autonomous systems. We prove Theorem \ref{thm_estimates_pas} concerning the behavior of the Partially Averaged System in Section \ref{section_estimates}. Section \ref{section_approximation} is devoted to proving Theorem \ref{thm_approximation_pas}, which provides a rigorous proof of what can be said about approximating the original system \eqref{FHN_pde} using the Partially Averaged System \eqref{FHN_pde_pas}. In the Appendix \ref{section_appendix}, we give details of the derivation of the Partially Averaged System \eqref{FHN_pde_pas} and other technical results.

\section{Well posedness result for non-autonomous systems}\label{section_well_posedness}
The goal of this section is to provide a proof for Theorem \ref{thm_well_posedness}. Throughout this paper, $f * g$ denotes the convolution in the space variable
$$(f * g) (x) = \int_\R f(x - y) g(y) dy.$$
In the case of vector-valued functions, this definition will be understood componentwise.

We define the norm in a product space as
$$\|(f,g)\|_{X\times Y} = \frac{1}{2}\|f\|_X + \frac{1}{2}\|g\|_Y$$
and as a shorthand we will denote $\|\cdot\|_{L^\infty\times L^\infty}$ as $\|\cdot\|_\infty$.

\subsection{Functional framework}\label{framework}
Concerning the functional framework where our results will hold, we consider a family of spaces similar to the one considered in \cite{rauch-smoller}. 
We work in some Banach spaces $\mathcal{B}$ contained in the space $BC^0(\R)$ of bounded uniformly continuous functions on $\R$. The precise conditions the space $\mathcal{B}$ has to satisfy are listed in the following definition.
\begin{defi}\label{space_rauch_smoller} 
A  Banach space $\mathcal{B}$ of functions $w:\R\rightarrow\R$ is admissible if the following conditions hold:
\begin{enumerate}
\item $\mathcal{B}$ is a subset of the bounded uniformly continuous functions $BC^0(\R)$  and for $w\in \mathcal{B}$, $\|w\|_\mathcal{B} \geq \|w\|_{L^\infty}$.
\item $\mathcal{B}$ is translation-invariant; i.e. if $W\in \mathcal{B}$, then $W \circ \tau \in \mathcal{B}$ for any translation $\tau: \R \to \R$, and $\|W \circ \tau\|_\mathcal{B} = \|W\|_\mathcal{B}$.
\item If $\tau_h:\R\to\R$ is the translation by $h$, i.e., $\tau_h(x) = x+h$, and $\tau_h w = w \circ \tau_h$, then for any $w\in \mathcal{B}$, $$\lim_{h \to 0}\|\tau_h w-w\|_\mathcal{B} = 0.$$
\item For every vector-valued function $F(U,x,t) \in V(\mathcal{B})$ we have the following property: for any $M>0$ there exist some constants $k_1, k_2, k_3>0$ such that for all $x\in\R$ and $t\in [0,\infty)$ we have
\begin{equation}\label{condition_non_linearity}
\|V\|_{\mathcal{B}\times \mathcal{B}}\leq M \text{ and }\|W\|_{\mathcal{B}\times \mathcal{B}}\leq M \Rightarrow \|F(V,\cdot,t) - F(W,\cdot,t)\|_{\mathcal{B}\times \mathcal{B}} \leq k_1\|V - W\|_{\mathcal{B}\times \mathcal{B}}, \tag{H1}
\end{equation}
\begin{equation}\label{condition_non_linearity2}
\|V\|_{\infty} \leq M \Rightarrow \|F(V,\cdot,t) - F(0,\cdot,t)\|_{\mathcal{B}\times \mathcal{B}} \leq k_2\|V\|_{\mathcal{B}\times \mathcal{B}},\tag{H2}
\end{equation}
\begin{equation}\label{condition_non_linearity3}
\|V\|_{\infty}\leq M \text{ and }\|W\|_{\infty}\leq M \Rightarrow \|F(V,\cdot,t) - F(W,\cdot,t)\|_{\infty} \leq k_3\|V - W\|_{\infty}.\tag{H3}
\end{equation}

\end{enumerate}

\end{defi}
The conditions in this definition impose a reasonable framework for studying the heat equation on the whole line, which is the linear component of our system. For $\sigma>0$ let $g_\sigma(x,t)$ be the fundamental solution to the heat equation in $\R$
\begin{equation}\label{green_heat}
g_\sigma(x,t) = \frac{1}{\sqrt{4\pi \sigma t}} \exp{\left(\frac{-x^2}{4\pi \sigma t}\right)}, \text{ for }t >0, \text{ and }g_0(x,t) = \delta(x).
\end{equation}
Then, for the operator $ \mathcal{L}= \partial_t -A\triangle$ with $A=diag(1,\rho)$, the corresponding Green's function is $G(x,t):=(g_1(x,t),g_\rho(x,t))$.

If $\mu$ is a finite Borel measure on $\R$, and $w\in \mathcal{B}$ for an admissible space, then $\mu*w\in \mathcal{B}$ and
$$\|\mu*w\|_\mathcal{B}\leq (\text{total variation of }\mu)\|w\|_\mathcal{B}$$
because of the translation invariance of the norm in $\mathcal{B}$. The most important application for us will be when $\mu_t = g_\sigma(x,t)dx$ in which case (total variation of $\mu_t$)$=\|g_\sigma(\cdot,t)\|_{L^1}=1$ (for all $t> 0$ and $\sigma>0$) and consequently
\begin{equation}\label{convolution_invariance}
\|g_\sigma(\cdot,t) * w\|_\mathcal{B}\leq \|w\|_\mathcal{B},
\end{equation}
for $\sigma >0$. The case $\sigma = 0$ is immediate.
\begin{remark}
Our assumptions in the nonlinearity $F(V,x,t)$ imply the following. Let $G$ be the Green's function of the operator $ \mathcal{L}U= \partial_t U -A\triangle U$ where $U=(u_1, u_2)$ and $A =  diag(1,\rho)$. Then, using the observation right above
\begin{align*}
H_s &= \sup_{t\in [s,s+1]}\left\|\int_s^t G(t-\tau-s,\cdot)*F(0,\cdot,\tau) d\tau\right\|_{\mathcal{B}\times \mathcal{B}}\\
& \leq \sup_{t\in [s,s+1]}\int_s^t \left\| G(t-\tau-s,\cdot)*F(0,\cdot,\tau)  \right\|_{\mathcal{B}\times \mathcal{B}} d\tau\\
&\leq \sup_{t\in [s,s+1]}\int_s^t \|F(0,\cdot,\tau) \|_{\mathcal{B}\times \mathcal{B}}d\tau  \\ 
&\leq \sup_{t\in [s,s+1]}\|F(0,\cdot,t) \|_{\mathcal{B}\times \mathcal{B}} \\
&\leq \sup_{t\in [0,\infty)}\|F(0,\cdot,t) \|_{\mathcal{B}\times \mathcal{B}} =:C^*,
\end{align*}
where the constant $C^*$ independent of $s>0$. For convenience, we will label this fact for later use
\begin{equation}\label{condition_non_linearity5}
\sup_{t\in [s,s+1]}\left\|\int_s^t G(t-\tau-s,\cdot)*F(0,\cdot,\tau)d\tau \right\|_{\mathcal{B}\times \mathcal{B}} < C, \text{ for }C>0 \text{ independent of } s>0.\tag{H4}
\end{equation}
\end{remark}
This notion of admissible space is not empty and includes relevant sets of functions, as the following result states.
\begin{proposition}\label{prop_hip_local_existence_PAS}
In addition to $BC^0(\R)$, the spaces $\mathcal{B} = W^{k,p}(\R)$, with $k\geq 1$ a positive integer, $p\geq 1$ a real, $kp >1$ and $\mathcal{B} = BC^0(\R)\cap L^p(\R)$ with $p \geq 1$ a real, are also admissible.
\end{proposition}

\begin{proof}
See Appendix \ref{proof_prop_hip_local_existence_PAS}.
\end{proof}

\subsection{Well-posedness result}
Let us start this subsection by defining a relevant notion introduced in \cite{rauch-smoller}.
\begin{defi}[Contracting Rectangle]\label{defi_invariant_rectangles}
For $L,S >0$ let $R_{L,S}$ be the rectangle centered at $(0,0)$ defined as
\begin{equation*}
 R_{L,S}= [- L,L]\times [-S,S].
\end{equation*}
Given a vector-value function $H: \R^2\times\R\times [0,T] \to \R^2$ we say that the rectangle $R_{L,S}$ is contracting under the vector fields $H$ if for each $\vec{v} \in \partial  R_{L,S}$ and every outward normal unit vector $n(\vec{v})$ we have
\begin{equation*}
\sup_{x\in\R}n(\vec{v}) \cdot H(\vec{v},x,t) < 0, \quad t\in [0,T]
\end{equation*}
(if $\vec{v}$ is in the corner of the rectangle, we assume this is satisfied for all $n(\vec{v})$ in the closed cone outward normal to the boundary).
The function $H$ can be understood as a vector field that can change as we move in $(x,t)$.
\end{defi}
This notion is useful to study the well-posedness of the following equation \eqref{eqn_generalized_parabolic} \begin{equation*}
\partial_t U = A  \partial_x^2U + F(U,x,t),
\end{equation*}
where $U=(u_1, u_2)$, $A= \text{diag}(1,\rho)$, $\rho \geq 0$ and $F$ is a smooth $\R^2$ valued function. Using Green's function $G(x,t) = (g_1(x,t), g_\rho(x,t))$ we can write this problem as an integral equation 

\begin{equation}\label{equation_fixed_point}
U(x,t) = G(t-s) * U(x,s) + \int_s^t G(t-\tau-s) * F(U(\tau),\cdot , \tau) d\tau.
\end{equation}
For our application, we have to consider a nonlinear term $F(U,x,t)$ that might depend explicitly on $x$ and $t$, and therefore, the local and global existence results in \cite{rauch-smoller} do not apply directly. For this purpose, we impose the additional conditions on the function $F$ given by \eqref{condition_non_linearity}, \eqref{condition_non_linearity2}, \eqref{condition_non_linearity3} and \eqref{condition_non_linearity5}. The next result is a generalization of \cite[Theorem 2.1]{rauch-smoller}  to non-autonomous systems. The following result gives us part (i) of Theorem \ref{thm_well_posedness}.
\begin{theorem}
[Local result]\label{lemma_local_existence}
Let $\mathcal{B}$ an admissible Banach space and $F(V,x,t): \R^2\times \R\times \R^+\to \R^2$ be a nonlinear vector-valued function in $V(\mathcal{B})$. For any $U_0 \in \mathcal{B}\times \mathcal{B}$, there exists a constant $t^*>0$, depending only on $F$ and $\|U_0\|_\infty$ such that the initial value problem for \eqref{eqn_generalized_parabolic} with data $U(t_0) = U_0$ has a unique solution $U$ in $C([t_0,t_0+t^*];\mathcal{B}\times \mathcal{B})$.
\end{theorem}
\begin{proof}
We first show that there exists $0<t^*<1$ depending only on $\|U_0\|_{\mathcal{B}\times \mathcal{B}}$ and $F$ such that \eqref{eqn_generalized_parabolic} has a unique solution in $C([t_0,t_0+t^*];\mathcal{B}\times \mathcal{B})$ and $\|U\|_{C([t_0,t_0+t^*];\mathcal{B}\times \mathcal{B})}\leq 2\|U_0\|_{\mathcal{B}\times \mathcal{B}}+2H_{t_0}$ where 
$$H_{t_0} = \sup_{t\in [t_0,t_0+1]}\left\|\int_{t_0}^{t} G(t-s-t_0)*F(0,\cdot,s) ds\right\|_{\mathcal{B}\times \mathcal{B}},$$
which is finite because of \eqref{condition_non_linearity5}. For any $t^* >0$, let us define the set
\begin{multline*}
\Omega = \Big\{ U\in C([t_0,t_0+t^*];\mathcal{B}\times \mathcal{B}) : \\
\left\|U(t)- G(t-t_0)*U_0 - \int_{t_0}^{t} G(t-s-t_0)*F(0,\cdot,s) ds\right\|_{\mathcal{B}\times \mathcal{B}} \leq \|U_0\|_{\mathcal{B}\times \mathcal{B}}  + H_{t_0}, t_0\leq t \leq t_0+t^* \Big\}.
\end{multline*}
If $U \in\Omega$, because of \eqref{convolution_invariance} we know $\|U(t)\|_{\mathcal{B}\times \mathcal{B}}\leq 2 \|U_0\|_{\mathcal{B}\times \mathcal{B}} + 2 H_{t_0} $, for $t_0\leq t \leq t_0+t^*$, so by condition \eqref{condition_non_linearity} there exists a constant $k>0$, independent of $t^*$ such that for $U,V\in \Omega$
\begin{equation}\label{analog28}
\|F(U(t),x,t) - F(V(t),x,t)\|_{\mathcal{B}\times \mathcal{B}}\leq k\|U(t)-V(t)\|_{\mathcal{B}\times \mathcal{B}}.
\end{equation}
Let $t^* = \frac{1}{2k}$, so that $t^*$ clearly depends only on $F$ and $\|U_0\|_{\mathcal{B}\times \mathcal{B}}$. We define the map $\Gamma$ from $C([t_0,t_0 +t^*];\mathcal{B}\times \mathcal{B})$ into itself by
$$\Gamma U(t) = G(t-t_0)*U_0 + \int_{t_0}^t G(t-s-t_0)*F(U(s),\cdot,s)ds.$$
We first show that $\Gamma$ maps the closed set $\Omega$ into itself. Using \eqref{convolution_invariance} and \eqref{analog28} with $V=0$, we have, for $U\in\Omega$
\begin{equation*}
\left\|\Gamma(U)(t) - G(t-t_0)*U_0 - \int_{t_0}^t G(t-s-t_0)*F(0,\cdot,s) ds\right\|_{\mathcal{B}\times \mathcal{B}}\leq k\int_{t_0}^t \|U(s)\|_{\mathcal{B}\times \mathcal{B}} ds.
\end{equation*}
Hence using $\|U(t)\|_{\mathcal{B}\times \mathcal{B}}\leq 2\|U_0\|_{\mathcal{B}\times \mathcal{B}} +  2H_{t_0}$ we obtain for $t_0\leq t \leq t_0+t^*$,
\begin{equation*}
\left\|\Gamma(U)(t)-G(t-t_0)*U_0 - \int_{t_0}^t G(t-s-t_0)*F(0,\cdot,s) ds\right\|_{\mathcal{B}\times \mathcal{B}}\leq k t^*( 2\|U_0\|_{\mathcal{B}\times \mathcal{B}}+ 2H_{t_0}) =\|U_0\|_{\mathcal{B}\times \mathcal{B}} +H_{t_0}
\end{equation*}
so that $\Gamma$ maps $\Omega$ into itself. Next, we show that $\Gamma$ is a contraction mapping on $\Omega$. If $U,V\in \Omega$, then
\begin{align*}
I& =\|\Gamma(U)(t)-\Gamma(V)(t)\|_{\mathcal{B}\times \mathcal{B}}\\
&\leq\int_{t_0}^t\|G(t-s-t_0)*(F(U(s),\cdot,s) -F(V(s),\cdot,s) )\|_{\mathcal{B}\times \mathcal{B}} ds\\
&\leq \int_{t_0}^t \|F(U(s),\cdot,s) -F(V(s),\cdot,s)\|_{\mathcal{B}\times \mathcal{B}} ds\\
&\leq k \int_{t_0}^t \|U(s)-V(s)\|_{\mathcal{B}\times \mathcal{B}} ds\\
&\leq k t^* \|U-V\|_{C([t_0,t_0+t^*];\mathcal{B}\times \mathcal{B})}\\
&\leq \frac{1}{2}\|U-V\|_{C([t_0,t_0+t^*];\mathcal{B}\times \mathcal{B})}.
\end{align*}
Taking the supremum on $t\in [t_0,t_0+t^*]$, we obtain that $\Gamma$ is a contraction and therefore has a unique fixed point in $\Omega$. However, this still leaves the possibility of finding solutions outside $\Omega$. The following proposition covers the property of uniqueness in a general time horizon.
\begin{proposition}[Uniqueness]
The solution of \eqref{eqn_generalized_parabolic} is unique in $C([t_0,t_0+T];\mathcal{B}\times \mathcal{B})$.
\end{proposition}
\begin{proof}
The argument is analog to \cite[Theorem 2.3]{rauch-smoller}. Let $U$, $\tilde{U}\in C([t_0,t_0+T];\mathcal{B}\times \mathcal{B})$ be two solutions of \eqref{equation_fixed_point}. Thus we have
\begin{equation*}
U(t)-\tilde{U}(t) = G(t-t_0) *(U(0)-\tilde{U}(0)) + \int_{t_0}^t G(t-s-t_0) * \left(F(U(s),\cdot,s) - F(\tilde{U}(s),\cdot,s)\right)ds. 
\end{equation*}
Then, if $M=\max \{\|U\|_{C([t_0,t_0+T];\mathcal{B}\times \mathcal{B})}, \|\tilde{U}\|_{C([t_0,t_0+T];\mathcal{B}\times \mathcal{B})} \}$ we get from \eqref{condition_non_linearity} that there exists $\tilde{k}_1=\tilde{k}_1(M)$ such that
$$\|U(t)-\tilde{U}(t)\|_{\mathcal{B}\times \mathcal{B}} \leq \|U(0)-\tilde{U}(0)\|_{\mathcal{B}\times \mathcal{B}} + \tilde{k}_1 \int_{t_0}^t \left\|U(s) -\tilde{U}(s)\right\|_{\mathcal{B}\times \mathcal{B}} ds $$
and by Gronwall's inequality, we obtain 
$$\|U(t) - \tilde{U}(t)\|_{\mathcal{B}\times \mathcal{B}} \leq e^{\tilde{k}_1 (t-t_0)} \|U(0)-\tilde{U}(0)\|_{\mathcal{B}\times \mathcal{B}},$$
which imply the uniqueness on $C([t_0,t_0+T];\mathcal{B}\times \mathcal{B})$. 
\end{proof}
To complete the proof of Theorem \ref{lemma_local_existence} we must extend the solution to an interval $t_0\leq t \leq t_0+t_1$ for a $0<t_1<1$ which depends on $\|U_0\|_\infty$ instead of $\|U_0\|_{\mathcal{B}\times \mathcal{B}}$. Since $\mathcal{B}\subset BC^0$, condition \eqref{condition_non_linearity3} and the above argument shows that there is a $t_1$ depending only on $\|U_0\|_\infty$ and $F$ and a solution $V\in C([t_0,t_0+t_1];BC^0\times BC^0)$ with $\|V(t)\|_\infty\leq 2 \|U_0\|_\infty + 2H_{t_0}$ for $t\in [t_0,t_0+t_1]$. By uniqueness of solution in $C([t_0,t_0+\min\{t^*,t_1\}];BC^0\times BC^0)$ we have $U=V$ for $t_0\leq t \leq t_0+ \min\{t^*,t_1\}$. If $t_1\leq t^*$ then the proof is completed since $V\in C([t_0,t_0+t_1];\mathcal{B}\times \mathcal{B})$ is the desired solution, so let us assume $t_1>t^*$.

In this last step, we will show that $V$  actually inherits the regularity of $U$ beyond $t^*$ and indeed $V \in C([t_0,t_0+t_1]; \mathcal{B}\times \mathcal{B})$. To prove the regularity of $V$ we will show that there exists an $\eta >0$ independent of $t_2\in [t_0,t_0+t_1]$ with the property that if $V\in C([t_0,t_0+t_2]; \mathcal{B}\times \mathcal{B})$ then $V\in C([t_0,\min\{t_0+t_2+\eta,t_1\}]; \mathcal{B}\times \mathcal{B})$. 
A finite number of applications of this result implies the regularity of $V$ all the way to $t_1$. The existence of $\eta$ will be achieved by obtaining an estimate for $\|V\|_{C([t_0,t_0+t_2]; \mathcal{B}\times \mathcal{B})}$ which is independent of $t_2$.

Let $t_2\in[t_0,t_0+t_1]$ and assume that
$V\in C([t_0,t_0+t_2]; \mathcal{B}\times \mathcal{B})$, $V$ is a solution of the integral equation  
$$V(t) = G(t-t_0)*U_0 + \int_{t_0}^t G(t-s-t_0)*F(V(s),\cdot,s)ds.$$
Taking norms on both sides yields
$$\|V(t)\|_{\mathcal{B}\times \mathcal{B}}\leq\|U_0\|_{\mathcal{B}\times \mathcal{B}}+\int_{t_0}^t \|F(V(s),\cdot,s) - F(0,\cdot,s)\|_{\mathcal{B}\times \mathcal{B}} ds + H_{t_0}.$$
By using \eqref{condition_non_linearity2} and the fact that $\|V(t)\|_\infty \leq 2\|U_0\|_\infty + 2H_{t_0}$ there is a $k_2$ so that 
$$\|F(V(s),\cdot,s)- F(0,\cdot,s)\|_{\mathcal{B}\times \mathcal{B}} \leq k_2\|V(s)\|_{\mathcal{B}\times \mathcal{B}}.$$ 
Gronwall's inequality applied to
\begin{align*} \|V(t)\|_{\mathcal{B}\times \mathcal{B}}
&\leq \|U_0\|_{\mathcal{B}\times \mathcal{B}}+ H_{t_0}+k_2 \int_{t_0}^t \|V(s)\|_{\mathcal{B}\times \mathcal{B}} ds 
\end{align*}
together with \eqref{condition_non_linearity5} yield a constant $C>0$, independent of $t_2\in [t_0,t_0+t_1]$, such that $\|V(t)\|_{\mathcal{B}\times \mathcal{B}}\leq C$ for $t_0\leq t \leq t_0+t_2$. In particular $\|V(t_2)\|_{\mathcal{B}\times \mathcal{B}}\leq C$.

From the proof of existence for $U$ in $[t_0,t^*]$, there exists $\eta >0$ only dependent on $C$ and $F$, and there exists $W\in C([t_2,\min\{ t_2+ \eta,t_1\}]; \mathcal{B}\times \mathcal{B})$ solution to the initial value problem \eqref{eqn_generalized_parabolic} starting at
$t_2$, with initial condition $W_0=V(t_2)$. Extend $V$ beyond $t_2$ as $V(t_2+s) := W(t_2+s)$ for $0\leq s \leq \min\{\eta,t_1-t_2\}$, implying that $V\in C([t_0,\min\{t_0+t_2+\eta,t_1\}]; \mathcal{B}\times \mathcal{B})$. Since $\eta>0$ independent of $t_2$ allows us to continue this process until reaching $t_1$, the proof of Theorem \ref{lemma_local_existence} is completed.
\end{proof}

Theorem \ref{lemma_local_existence} tells us that if we have constructed a solution $U$ of \eqref{eqn_generalized_parabolic} for $t\in [0,T]$, this solution can be extended to the interval $[0,T+t^*]$ where $t^*$ only depends on $F$ and $\|U(T)\|_\infty$. If in this process $\|U\|_\infty$ remains uniformly bounded in time, then the solution can be continued for all $t\in[0,\infty)$, which is what we will do in the following theorem under some additional assumptions. The following result is an extension of \cite[Theorem 3.9]{rauch-smoller}  to the non-autonomous case and gives part (ii) of Theorem \ref{thm_well_posedness}.
\begin{theorem}
[Global result]\label{lemma_global_existence}
Assume that, in addition to being continuous, the functions in $\mathcal{B}$ also decay to 0 at $\pm \infty$. Assume that $U_0\in \mathcal{B}\times \mathcal{B}$ and that there exists some $\hat{\epsilon} >0$ and a rectangle $R_{L,S}$ such that 
$$(1+\hat{\epsilon})U_0(x) \in R_{L,S}, \quad \forall x\in\R,$$ 
where the rectangle $R_{L,S}$ is contracting under $F(U,x,t)$ for all $t >0$. Then, there is a unique solution $U\in C([0,\infty);\mathcal{B} \times \mathcal{B})$ of \eqref{eqn_generalized_parabolic} with $U(0) = U_0$. Additionally $U$ also satisfies
$$U(x,t) \in R_{L,S},\quad \forall x\in\R, t\geq 0.$$
\end{theorem}

\begin{proof}
Let $R_{L,S}$ be a rectangle satisfying the hypothesis of this theorem. Define a  norm in the space 
\begin{equation*}
X= \{\vec{v}\in C(\R;\R^2) : \vec{v} \text{ is a continuous function that} \text{ converges to }(0,0) \text{ when }x \to \pm \infty\},
\end{equation*}
given by
\begin{equation}\label{definition_norm_X}
\|\vec{v}\|_X = \sup_{x\in\R}\inf\{r>0: \vec{v}(x)\in r R_{L,S}\}.
\end{equation}
If $W\in \mathcal{B}\times \mathcal{B}$, observe that $\|W\|_X<1$ is equivalent to $[~\exists \hat{\epsilon}>0: (1+\hat{\epsilon})W(x)\in R_{L,S}, \forall x\in\R~]$. We also have that $\|W\|_X<1$ implies $\|W\|_\infty\leq (L+S)/2$, so the assumptions in $U_0$ imply $\|U_0\|_X<1$ and $\|U_0\|_\infty\leq (L+S)/2.$

Fix $t_0=0$ and let $U\in C([0,t^*];\mathcal{B}\times \mathcal{B})$ be the local solution given by Theorem \ref{lemma_local_existence} with $U(0)=U_0$, noting that $t^*$ can be chosen depending only on $F$ and $(L+S)/2$. Define 
$$E(t) = \|U(\cdot,t)\|_X$$
and observe that $E(0) < 1$. Let us look at the upper Dini derivative of $E$, defined as
$$\bar{D} E(t) = \limsup_{h\to 0}\frac{E(t+h)-E(t)}{h}, \text{ for } t\in(0,t^*).$$
Since $\partial R_{L,S}$ and the set of outward pointing unit normals to $R_{L,S}$ at the boundary are compact, then $R_{L,S}$ being contracting implies that there exists $\eta >0$ such that $\sup_{x\in\R}H(W,x,\hat{t}) \cdot n(W) < -\eta$ for any $W\in \partial R_{L,S}$ and $n(W)$ outward unit vector to $\partial R_{L,S}$ at $W$. Therefore, if $E(\hat{t})=1$ then \cite[Lemma 3.8]{rauch-smoller} (the condition that functions in $\mathcal{B}$ vanish at $\pm\infty$ is used here) would imply
$$\bar{D} E(\hat{t}) \leq -\frac{2\eta}{\min\{2L, 2S\}} E(\hat{t})<0.$$
Since the continuous quantity $E(t),~ t\in[0,t^*)$, satisfies $E(0)<1$ and that $\bar{D} E(t)<0$ if ever $E(t)=1$, then it follows that $E(t)<1$ for all $t\in [0,t^*)$. In particular, $E(\frac{t^*}{2})=\|U(\cdot, \frac{t^*}{2})\|_X <1$. We can now repeat the same argument above, but starting at $t_0=\frac{t^*}{2}$, obtaining that the solution $U$ can be extended to exist in all of $[0,\frac{t^*}{2}+t^*]$ and that $E(t)<1$ $\forall t \in [0,\frac{3t^*}{2})$. And we can repeat this argument ad infinitum, advancing on steps of size $\frac{t^*}{2}$ each time, concluding that the solution $U$ can be extended to exist for all $t \in [0,\infty)$.
\end{proof}

\section{Stability for the partially averaged system}\label{section_estimates}
The goal of this section is to prove Theorem \ref{thm_estimates_pas}. This will be done by showing that some suitable small rectangles are contracting with respect to the vector field associated with the Partially Averaged System \eqref{FHN_pde_pas} and applying Theorem \ref{thm_well_posedness}. Some technical conditions are required on the set of parameters. This is described in the following result.

\begin{lemma}

\label{properties_admissible_parameters}
Let us recall that $\beta, \gamma >0 $ are said to be admissible if they satisfy \eqref{condition_H1}. The set of admissible parameters is non-empty, and for any $0<\delta<1/4$, it contains the set
\begin{equation}\label{set_inclusion_parameters}
\Big\{(\beta,\gamma): \gamma \geq \frac{2\delta +1}{\delta}, \beta \geq \frac{2}{3}\gamma\Big\}.\end{equation}
Additionally, if $\beta$, $\gamma$ are admissible then we have the following:
\begin{enumerate}
\item System \eqref{defi_v0w0}  has a unique solution $(v_0, w_0)$. \label{lemma_admis_1}
\item For $\beta,\gamma$ admissible, let $\delta>0$ be such that \eqref{condition_H1} is fulfilled. Then the corresponding solution $(v_0, w_0)$ satisfies the bounds \label{lemma_admis_2}
\begin{equation}\label{condition_v0}
\min \{-\beta, -\sqrt{3}\}\leq v_0 <-\sqrt{1+\frac{1}{\delta\gamma}}<0.
\end{equation}
These bounds readily imply that
$$\frac{1}{\max\{\beta^2-1,2\}}\leq \frac{1}{(v_0)^2-1} < \delta \gamma.$$
\end{enumerate}
\end{lemma}\begin{proof}
See Appendix \ref{app_properties_admissible_parameters}
\end{proof}

\subsection{Existence of contracting rectangles}
In order to study system \eqref{FHN_pde_pas} we have to consider the following vector valued function $H:\R^2\times\R\times[0,\infty)\to\R^2$,
\begin{align}\label{def_vector_field_H}
H((V,W),x,t) &= \binom{H_1((V,W),x,t)}{H_2((V,W),x,t)}\\
H_1((V,W),x,t) &=
(1-v_0^2) V  -v_0 V^2 - \frac{V^3}{3}  - W - (\frac{A(x)^2}{2}+\frac{B(x)^2}{2} + A(x) B(x) \cos(\eta t))(V+v_0) \notag \\
H_2((V,W),x,t) &= \varepsilon(V - \gamma W)  \notag
\end{align}
which clearly belongs to $V(\mathcal{B})$ given by Definition \ref{defi_space_vB} for the spaces $\mathcal{B}$ in Proposition \ref{prop_hip_local_existence_PAS}. The following result is key to addressing the behavior of the solutions of \eqref{FHN_pde_pas}.
 
\begin{lemma}[Existence of small Contracting Rectangles]\label{lemma_invariant_rectangles} Let $\varepsilon >0$, let $\beta$,  $\gamma$ satisfy \eqref{condition_H1} and let $(v_0,w_0)$ be the unique solution of \eqref{defi_v0w0}. Given $A(x),~B(x)$ recall the definition in equation \eqref{decay_AxBx} and define 
$$\Delta = \Delta(A,B):=\sup_x|A(x)|+ \sup_x|B(x)| = \frac{|a|}{d_1^2} + \frac{|b|}{d_2^2}.$$ 
There exists $\Delta^*>0$ such that for any $A,B$ satisfying $0\leq \Delta\leq \Delta^*$ there is a non empty set $D(\Delta)\subset\R_+^2$ such that for all $(L,S)\in D(\Delta)$ the rectangles $R_{L,S}$
are contracting under the flow of $H((V,W),x,t)$ given by \eqref{def_vector_field_H}. Additionally, the set $D(\Delta)$ satisfies the inclusion
\begin{equation}\label{SetLS}
D(\Delta) \supseteq \biggl\{(L,S): 0<L<|v_0|, ~~\frac{1}{\gamma} < \frac{S}{L} < v_0^2-1, L < \frac{(v_0^2-1-S/L)}{\Delta^2(-v_0-L)/L^2+(-v_0-L/3)}\biggl\}.
\end{equation}
Moreover, given any open neighborhood $\mathcal{O}$ of $(0,0)$ there exists $N>0$ such that for all $\Delta \leq N$ there exists $(L,S) \in D(\Delta)$ such that $R_{L,S}\subset \mathcal{O}$. 
\end{lemma}
\begin{remark}
In the case $\Delta = 0$, the existence of arbitrarily small contracting rectangles for $v_0^2-1>\frac{1}{\gamma}$ was obtained in the article \cite{rauch-smoller}.
\end{remark}

\begin{proof}
We want to show that if $\Delta:=\sup_x|A(x)| +\sup_x|B(x)|$  is small enough, then the set of pairs $(L,S)\in \R_+^2$ such that the rectangle $R_{L,S}$ is contracting is non-empty. For this purpose, we have to verify that the vector field $H((V,W),x,t)$ is pointing inwards at each point of the boundary $\partial R_{L,S}$. On each face, we can write the outward pointing vector explicitly, and in the corners, it can only be a linear combination of the vector used in the adjacent faces. 
\begin{enumerate}
\item Top face. At $W = S$, $V\in [-L,L]$ we have
\begin{align*}
(0,1) \cdot H((V,W),x,t) &= \varepsilon( V- \gamma S)\\
&\leq \varepsilon (L-\gamma S).
\end{align*}
Therefore, the vector field will point inwards if $L-\gamma S <0$, or equivalently $\frac{1}{\gamma}<\frac{S}{L}$.
\item Bottom face. At $W =-S$, $V\in [-L,L]$ we have
\begin{align*}
(0,-1) \cdot H((V,W),x,t) &= -\varepsilon( V+ \gamma S)\\
&\leq \varepsilon (L-\gamma S).
\end{align*}
Therefore, we get the same condition as for the top face.
\item Left face. Assume $0< L<-v_0$. For $V = - L$, $W\in [-S,S]$ we get the following
\begin{align*}
I &= (-1,0) \cdot H((V,W),x,t)\\
&= -\Big((1-v_0^2)(-L) -v_0 L^2 - (-L)^3/3 - W - (v_0-L)(A(x)^2/2+B(x)^2/2+A(x)B(x)\cos(\eta t)) \Big)\\
&\leq -\left( (v_0^2-1)L - v_0 L^2+L^3/3 -\frac{(v_0-L)}{2}(|A(x)|-|B(x)|)^2 -S \right)\\
&= -\left((v_0^2 -1-\frac{S}{L})L+(-v_0 +L/3)L^2\right)-\left(\frac{(-v_0+L)}{2}(|A(x)|-|B(x)|)^2  \right)\\
&\leq -(v_0^2 -1-\frac{S}{L})L \quad(\text{ since } 0<L<-v_0).
\end{align*}
We want to choose $L$ and $S$ such that the right-hand side is negative, which will be the case if the following equation holds
\begin{equation}\label{condition_left_face}
v_0^2 -1-\frac{S}{L}>0.
\end{equation}
Because of \eqref{condition_v0} we have that $v_0^2 -1 -\frac{1}{\delta\gamma} >0$ for $\delta>0$ corresponding to the admissibility of $\beta,\gamma$. But since $\delta \leq 1$ (we are restricting to $\delta <1/4$) then $v_0^2 -1 -\frac{1}{\gamma} >0$. Therefore condition \eqref{condition_left_face} is satisfied if
\begin{equation}\label{condition_final_left_face}
0<L<-v_0 \text{ and }\frac{1}{\gamma} < \frac{S}{L} < v_0^2-1. 
\end{equation}
\item Right face. Assume $0<L<-v_0$. For $V = L$,  $W\in [-S,S]$ we get the following
\begin{align*}
(1,0) \cdot H &= (1-v_0^2)L -v_0 L^2 - L^3/3 - W-(v_0+L)(A(x)^2/2+B(x)^2/2+A(x)B(x)\cos(\eta t))\\
&\leq (1-v_0^2)L -v_0 L^2 - L^3/3 -\frac{(v_0+L)}{2}(|A(x)|+|B(x)|)^2 +S\\
&=(|A(x)|+|B(x)|)^2\frac{(-v_0-L)}{2}  +(-v_0 -L/3)L^2 - \left(v_0^2  -1-\frac{S}{L}\right)L\\
&\leq \Delta^2 (-v_0-L)  +(-v_0 -L/3)L^2  - \left(v_0^2  -1-\frac{S}{L}\right)L.
\end{align*}
Since $\Delta^2(-v_0-L)  +(-v_0 -L/3)L^2>0$, the right-hand side in the equation above will be negative if 
\begin{equation}\label{ineq_final_right_face}
L < \frac{(v_0^2-1-\frac{S}{L})}{ \Delta^2(-v_0-L)/L^2+(-v_0-L/3)}.
\end{equation}
\end{enumerate}
Combining \eqref{condition_final_left_face} and \eqref{ineq_final_right_face} we obtain the following set of sides $(L,S)$ for which $R_{L,S}$ is a contracting rectangle under the flow $H((V,W),x,t)$,
\begin{equation*}
D(\Delta) \supseteq \biggl\{(L,S): 0<L<|v_0|, ~~\frac{1}{\gamma} < \frac{S}{L} < v_0^2-1, L < \frac{(v_0^2-1-S/L)}{\Delta^2(-v_0-L)/L^2+(-v_0-L/3)}\biggl\}.
\end{equation*}
For the last part of the lemma let us consider $\mathcal{O}$ an arbitrary neighborhood of $(0,0)$ and take $\epsilon >0$ small so that $\frac{1+\epsilon}{\gamma} < v_0^2-1$ and set $S_0=\frac{1+\epsilon}{\gamma}L_0$. Next, take $0<L_0<|v_0|$ small enough so that $(\pm L_0,\pm S_0)\in \mathcal{O}$ and
$$ 0<L_0 <  \frac{(v_0^2-1-\frac{1+\epsilon}{\gamma})}{(-v_0-L_0/3)}.$$
Finally, notice that the limit 
$$\lim_{\Delta\to 0}\left(\frac{(v_0^2-1-S_0/L_0)}{\Delta^2(-v_0-L_0)/L_0^2+(-v_0-L_0/3)} - L_0 \right) = \frac{(v_0^2-1-\frac{1+\epsilon}{\gamma})}{(-v_0-L_0/3)} - L_0 > 0,$$
which implies that for $\Delta>0$ small enough the pair $(L,S)=(L_0,\frac{1+\epsilon}{\gamma}L_0)\in D(\Delta)$ satisfy $R_{L,S}\subset \mathcal{O}$. This concludes the proof of Lemma \ref{lemma_invariant_rectangles}.
\end{proof}

\subsection{Proof of Theorem \ref{thm_estimates_pas}}

Suppose the parameters $\varepsilon >0$ and assume $\beta$, $\gamma$ satisfy \eqref{condition_H1} and $H$ be the vector valued function given by \eqref{def_vector_field_H}.
From Lemma \ref{lemma_invariant_rectangles} we know that given a neighborhood $\mathcal{O}$ of $(0,0)$ there exists $N>0$ such that if 
$$\Delta= \frac{|a|}{d_1^2}+\frac{|b|}{d_2^2} < N$$
then there exists a contracting rectangle $R_{L,S} \subset \mathcal{O}$ for the vector valued function $H$.
Now, thanks to Proposition \ref{prop_hip_local_existence_PAS} we know that the space $\mathcal{B}$ satisfy conditions \eqref{condition_non_linearity}, \eqref{condition_non_linearity2}, \eqref{condition_non_linearity3} and \eqref{condition_non_linearity5}. Thus, we can apply Theorem \ref{thm_well_posedness} to conclude that if the initial condition satisfies  $(f_0-v_0, g_0-w_0)\in \mathcal{B}\times \mathcal{B}$ and for some $\epsilon >0$
$$(1+\epsilon)(f_0(x)-v_0,g_0(x)-w_0) \in R_{L,S}, \quad \forall x\in \R$$
 then there exists a unique solution $(V,W)\in C([0,\infty), \mathcal{B}\times \mathcal{B})$ for the initial value problem \eqref{FHN_pde_pas} with initial data $V(0) = f_0-v_0$, $W(0) = g_0-w_0$. Moreover, the solution satisfies
$$(V,W)\in R_{L,S}, \quad \forall x\in \R, t\geq 0.$$
This concludes the proof of Theorem \ref{thm_estimates_pas}
\qed

Some immediate properties for the solution $(V,W)$ are included in the following proposition.
\begin{proposition}[Properties of the solution of Partially Averaged System]\label{properties_solution_pas}
Let $\varepsilon >0$, and let $\beta$,  $\gamma$ satisfy \eqref{condition_H1}. Let $(v_0,w_0)$ be the unique solution of \eqref{defi_v0w0}. Suppose the hypothesis of Theorem \ref{thm_estimates_pas} holds for $\mathcal{B} = W^{k,p}(\R)$ with $k\geq 3$, $p\geq 1$ and for the initial condition $(V(0), W(0)) \in W^{k,p}(\R)\times W^{k,p}(\R)$. Let $(V,W) \in C([0,\infty);W^{k,p}(\R)\times W^{k,p}(\R))$ be the unique solution of \eqref{FHN_pde_pas} given by Theorem \ref{thm_estimates_pas}. Then, there exist positive constants $C_1, C_2, C_3$, independent of $\omega_1$ (dependent on $\|V(0)\|_{W^{k,p}}$ and $\|W(0)\|_{W^{k,p}}$) such that we have the following
\begin{align*}
|V(x,t)|\leq C_1, |\partial_x V(x,t)|\leq C_2, |\partial_t V(x,t)|\leq C_3, \forall x\in \R, t>0.
\end{align*}
\end{proposition}
\begin{proof}
For part $i)$, first we notice that Theorem \ref{thm_estimates_pas} tells us that $|V(x,t)|\leq C$ and $|W(x,t)|\leq C$ for all $t\in[0,\infty)$ and $x\in \R$. Second, since we have a solution in $C([0,1]; W^{k,p}(\R)\times W^{k,p}(\R))$ we know that it $W^{k,p}(\R) \times W^{k,p}(\R)$ norm remains bounded (this norm might be growing in $t$, but stays finite for $t\in [0,1]$). Next, thanks to the Sobolev embedding we know that $\|f\|_{C^{2,\gamma}} \leq C \|f\|_{W^{k,p}}$ and therefore $|\partial_x V(x,t)|$ and $|\partial_x^2 V(x,t)|$ are bounded for all $t\in [0,1]$ and $x\in \R$. Next, for any time $t \geq 1$ because the coefficient of the Laplacian in the first equation of \eqref{FHN_pde_pas} is nonzero, we can estimate spatial derivatives of $V$ by taking derivatives of the heat kernel in \eqref{equation_fixed_point}, which tells us that $\|\partial_x V(t)\|_{L_x^\infty}$ and $\|\partial_x^2 V(t)\|_{L_x^\infty}$ remain bounded for all $t\in [1,\infty)$. Lastly, to obtain the estimate in the time derivative, because we have enough regularity, the solution given by Theorem \ref{thm_estimates_pas} is a classical solution and therefore we can use the first equation in \eqref{FHN_pde_pas} to estimate $|\partial_t V (x,t)|$ in terms of $| V(x,t)|$, $| W(x,t)|$ and $|\partial_x^2 V(x,t)|$ and since each one of those quantities is uniformly bounded in time and space, we conclude that $|\partial_t V (x,t)|$ is also uniformly bounded. 
\end{proof}

\section{Proof of the approximation result}\label{section_approximation}
The goal of this section is to prove Theorem \ref{thm_approximation_pas}. This is done by studying the problem of the approximation error of using the Partially Averaged System \eqref{FHN_pde_pas} instead of \eqref{FHN_pde}, we proceed in two steps: 
\begin{itemize}
    \item[i.] We consider the linear part of the approximation error to obtain appropriate estimates that depend on the solution of the Partially Averaged System and the frequency $\omega_1$ used in the input current \eqref{input_current}. These estimates will impose some conditions in the parameters of the system, which are part of condition \eqref{condition_H1}.
    \item[ii.] Using the previous result, we study the full nonlinear approximation error equation and conclude that we have uniform estimates for the approximation error for all time under suitable conditions.
\end{itemize}

Since we will continue using the tool of contracting rectangles, which is well adapted to use uniform estimates, it is convenient to consider the following norm.
\begin{defi}
Let $f\in C([0,T];BC^0)$ we consider the norm
\begin{equation}\label{defi_normY}
\|f(x,t)\|_Y = \sup_{0<t<T}\sup_{x\in \R} |f(x,t)|.
\end{equation}

\end{defi}
Additionally, in this section the space $\mathcal{B}$ denote the space $\mathcal{B}=W^{k,p}(\R)$ with $k\geq 3$, $p\geq 1$. 

\begin{proposition}[Equation for the approximation error]\label{prop_eqn_approximation_error}
Let $(v,w)$ the solution of the centered FHN system \eqref{FHN_pde} and $(V,W)$ the solution of the Partially Averaged System \eqref{FHN_pde_pas} with the same initial data. Then, the approximation error given by $E_{v} = v - V - J_0(x,t)$, $E_{w} =v - W$ satisfy 
\begin{align}\label{full_error_system}
\left\{
\begin{array}{rl}
\partial_t E_v - \partial_x^2 E_v &= (1-(v_0+V)^2 +\varphi_1) E_v + \varphi_2E_v^2 - \frac{1}{3} E_v^3\\
&\hspace{1cm} -E_w+ \varphi_3,\\
\partial_t E_w -\rho\partial_{x}^2E_w&= \varepsilon \left(E_v - \gamma E_w\right)+ \varepsilon J_0,\\
E_v(0) =0, & E_w(0) = 0,
\end{array}
\right.
\end{align}
where 
\begin{align}
J_0(x,t) &= A(x) \sin(\omega_1 t) + B(x) \sin(\omega_2 t), \label{def_J0}\\
\varphi_{1} &= - J_0^2-2(v_0+V) J_0 , \label{def_varphi1}\\
\varphi_2 &= -(v_0+V+J_0), \label{def_varphi2}\\
\varphi_3 &= \partial_x^2 J_0 - J_0^3/3-(v_0+V)^2J_0 \notag\\
&\hspace{1cm}+(v_0+V)\biggl(\frac{A(x)^2}{2} \cos(2\omega_1 t) +\frac{B(x)^2}{2} \cos(2\omega_2 t) + A(x)B(x)\cos(\omega_1+\omega_2)t \biggl)\label{def_varphi3}.
\end{align}
\end{proposition}
\begin{proof}
It is immediate from taking the difference between \eqref{FHN_pde} and \eqref{FHN_pde_pas}.
\end{proof}

\subsection{Linear estimate of the error}\label{subsection_linear_estimate_error}

The first ingredient to prove the approximation result is to look at the following linear problem
\begin{equation}\label{equation_linear_estimate_error}
\left\{
\begin{array}{rl}
\partial_t F_v - \partial_{x}^2F_v &= (1-(v_0+V)^2 + \varphi_1) F_v - F_w + \varphi_3,\\
\partial_t F_w-\rho\partial_{x}^2F_w&= \varepsilon F_v - \varepsilon \gamma F_w + \varepsilon J_0,\\
F_v(x,0)=0, & F_w(x,0) = 0,
\end{array}\right.
\end{equation}
where $v_0$ is given by \eqref{defi_v0w0}, and $\varphi_1$, $\varphi_3$ and $J_0$ are given by \eqref{def_varphi1}, \eqref{def_varphi3} and \eqref{def_J0}, respectively. 
For the regularity of the solutions of the equation of the approximation error, we need to apply Theorem \ref{thm_well_posedness}. 

\begin{lemma}[Linear estimate of the error]\label{lemma_linear_estimate_error}
Let $\beta$, $\gamma$ satisfy \eqref{condition_H1}, let $\varepsilon>0$, and let $(v_0,w_0)$ be given by \eqref{defi_v0w0}. Suppose that the parameters $a$, $b$, $d_1$, $d_2$, $\gamma$ and the solution $(V,W)$ of Partially Averaged System \eqref{FHN_pde_pas} satisfy, for some $T>0$, the estimate
\begin{equation}\label{condition_linear_estimate}
\alpha(T):= \frac{\|v_0^2 - (v_0+V)^2\|_Y}{v_0^2-1}  +\frac{M^2}{v_0^2-1} +  \frac{2M\|v_0+V \|_Y }{v_0^2-1}   +
\frac{1}{\gamma (v_0^2-1)} <1, t\in [0,T]
\end{equation}
where $M = |a|/d_1^2+|b|/d_2^2 \geq \|J_0\|_Y$. Then, there exists $(F_v, F_w)$ $\in C([0,T],$ $\mathcal{B}\times \mathcal{B})$ solution of  the initial value problem \eqref{equation_linear_estimate_error} with initial data $(F_v(0),\!F_w(0))$ $=(0,0)$, and a constant $C>0$ independent of $\omega_1$ such that 
\begin{equation}\label{estimate_linear_error}
|F_v(x,t)| \leq \frac{C}{\omega_1}, \quad |F_w(x,t)|\leq \frac{C}{\omega_1},
\end{equation}
for all $x\in \R$, $0\leq t\leq T$ and $\omega_1 \geq 1$.
\end{lemma}

\begin{proof}
For the regularity we apply  Theorem \ref{lemma_local_existence} to guarantee that if we start at $(0,0)\in \mathcal{B}\times \mathcal{B}$ and we prove that  $\sup_{x\in\R}|F_v(x,t)| +\sup_{x\in\R}|F_w(x,t)|\leq C$ for $t\in [0,T]$, then the solution belongs to $C([0,T]; \mathcal{B}\times \mathcal{B})$. 

The idea of the proof is to consider an iterative approximation of \eqref{equation_linear_estimate_error} and obtain some decay on $\omega_1$ by looking at the highly oscillatory terms. To simplify the notation, let $g_\sigma(x,t)$ as in \eqref{green_heat} and notice that $g_\sigma dx$ is a measure of total mass 1 for all $\sigma \geq 0$. Multiplying the first equation in \eqref{equation_linear_estimate_error} by $e^{-(1-v_0^2)t}$ we get
\begin{equation*}
\partial_t e^{-(1-v_0^2)t}F_v - \partial_x^2 e^{-(1-v_0^2)t}F_v 
= ((v_0^2-(v_0+V)^2) + \varphi_{1}) e^{-(1-v_0^2)t}F_v - e^{-(1-v_0^2)t}F_w + e^{-(1-v_0^2)t}\varphi_3.
\end{equation*}
By virtue of Duhamel's principle, we get 
\begin{equation*}
F_v(x,t) = g_1(t) * F_v(\cdot,0)
+ \int_0^t  e^{(1-v_0^2)(t-\tau)} g_1(t-\tau)* \left(((v_0^2-(v_0+V)^2) + \varphi_{1})F_v - F_w + \varphi_3\right) d\tau.
\end{equation*}
For the second equation, we obtain 
\begin{align*}
F_w(x,t) = g_\rho(t)*F_w(\cdot,0) + \varepsilon\int_0^{t} e^{-\varepsilon \gamma(t-\tau)} g_\rho(t-\tau)* \left(F_v + J_0\right) d\tau.
\end{align*}
Let us consider the following iterative procedure. Set $F_v^{(0)} = 0$, $F_w^{(0)}=0$ and define
\begin{align*}
F_v^{(k+1)}(x,t) &= g_1(t)*F_v(\cdot,0)+\int_0^t e^{ (1-v_0^2)(t-\tau)} g_1(t-\tau) *  \left(((v_0^2-(v_0+V)^2) + \varphi_{1})F_v^{(k)} - F_w^{(k)} + \varphi_3\right) d\tau\\
F_w^{(k+1)}(x,t) &= g_\rho(t) * F_w(\cdot,0) + \varepsilon\int_0^{t} e^{-\varepsilon \gamma ( t-\tau)} g_\rho(t-\tau)*\left(F_v^{(k+1)} + J_0\right) d\tau,
\end{align*}
where $\varphi_1$, $\varphi_3$ and $J_0$ are defined by \eqref{def_varphi1}, \eqref{def_varphi3} and \eqref{def_J0}, respectively, and where $(F_v^{(k)}, F_w^{(k)}) \in  BC^0\times  BC^0$ imply that $(F_v^{(k+1)}, F_w^{(k+1)}) \in  BC^0\times  BC^0$. The next step is to look at the convergence of the sequences $\{F_v^{(k)}\}$, $\{F_w^{(k)}\}$ in the space $C([0,T]; BC^0)$. We consider the norm in $C([0,T]; BC^0)$ given by \eqref{defi_normY} and estimate the difference between two consecutive terms. For $\{F_v^{(k)}\}$, it holds
\begin{align*}
&I= \|F_v^{(k+1)}-F_v^{(k)}\|_Y\\
&\leq \sup_{0<t<T}\int_0^t e^{ (1-v_0^2)(t-\tau)} \left\|g_1(t-\tau)\!*\! \left(\left|v_0^2-(v_0+V(\cdot,\tau))^2\right| + |\varphi_{1}(\cdot,\tau)|\right)\right\|_{L^\infty_x} d\tau \|F_v^{(k)}-F_v^{(k-1)}\|_Y\\
&~~+ \sup_{0<t<T}\int_0^t e^{(1-v_0^2)(t-\tau)} d\tau   \left\|F_w^{(k)}-F_w^{(k-1)} \right\|_Y\\ 
&\leq \sup_{0<t<T}\int_0^t e^{ (1-v_0^2)(t-\tau)}\left\|g_1(t-\tau)\!*\! \left(\left|v_0^2-(v_0+V(\cdot,\tau))^2\right| + |\varphi_{1}(\cdot,\tau)|\right)\right\|_{L^\infty_x} d\tau  \|F_v^{(k)}-F_v^{(k-1)}\|_Y\\
&~~+ \sup_{0<t<T}\int_0^t e^{ (1-v_0^2)(t-\tau)}d\tau  \frac{1}{\gamma}\|F_v^{(k)} - F_v^{(k-1)}\|_Y.
\end{align*}
Since $\varphi_1= J^2_0-2(v_0+V)J_0$, 
\begin{align*}
  J&= \left\|g_1(t-\tau)* \left(\left|v_0^2-(v_0+V(\cdot,\tau))^2\right| + |\varphi_{1}(\cdot,\tau)|\right)\right\|_{L^\infty_x}\\
  &\leq \left\| \left|v_0^2-(v_0+V(\cdot,\tau))^2\right| + |\varphi_{1}(\cdot,\tau)|\right\|_{L^\infty_x}\\
    &\leq  \left\| v_0^2-(v_0+V(\cdot,\tau))^2 \right\|_{L^\infty_x} +\left\| J^2_0 \right\|_{L^\infty_x}+2 \left\| (v_0+V)J_0\right\|_{L^\infty_x}\\
    &\leq  \left\| v_0^2-(v_0+V(\cdot,\tau))^2 \right\|_Y + M^2+2M\|v_0+V\|_Y
\end{align*}
utilizing this bound in the estimate for $\|F_v^{(k+1)}-F_v^{(k)}\|_Y$ gives
\begin{align*}
I&= \|F_v^{(k+1)}-F_v^{(k)}\|_Y \\
&\leq  \|F_v^{(k)} - F_v^{(k-1)}\|_Y\left( \left\| v_0^2-(v_0+V(\cdot,\tau))^2 \right\|_Y + M^2+2M\|v_0+V\|_Y +\frac{1}{\gamma} \right)\\
&\hspace{1cm}\times \sup_{0\leq t\leq T}\int_0^t e^{(1-v_0^2)(t-\tau)}d\tau\\
&\leq  \|F_v^{(k)} - F_v^{(k-1)}\|_Y \frac{\left\| v_0^2-(v_0+V(\cdot,\tau))^2 \right\|_Y + M^2+2M\|v_0+V\|_Y +\frac{1}{\gamma} }{v_0^2-1}
\end{align*}
where we used that $v_0^2-1>0$. Defining $\alpha(T) = (\| v_0^2-(v_0+V(\cdot,\tau))^2\|_Y + M^2+2M\|v_0+V\|_Y +\frac{1}{\gamma} )/(v_0^2-1)$ we obtained that 
\begin{equation*}
\Vert F_v^{(k+1)}-F_v^{(k)}\Vert_Y \leq \alpha(T)\Vert F_v^{(k)}-F_v^{(k-1)}\Vert_Y .
\end{equation*}
Analogously, for $\{F_w^{(k)}\}$, we get
\begin{align*}
      \Vert F_w^{(k+1)}-F_w^{(k)}\Vert_Y &\leq \varepsilon \sup_{0<t<T}\int_0^te^{-\varepsilon \gamma(t-\tau)}d\tau\Vert F_v^{(k+1)}-F_v^{(k)}\Vert_{Y}\\
      &\leq 
      \frac{1}{\gamma} \Vert F_v^{(k+1)}-F_v^{(k)}\Vert_{Y} .
\end{align*}
Now, to state the convergence of the sequences $\{F_v^{(k)}\}$, $\{F_w^{(k)}\}$ in $C([0,T];$ $BC^0)$, use that for $m> n$
\begin{equation*}
\|F_v^{(m)} -F_v^{(n)} \|_Y \leq  \sum_{k=n}^{m-1}\left\|F_v^{(k+1)} - F_v^{(k)}\right\|_{Y}\leq \sum_{k=n}^{m-1} \alpha(T)^k \|F_v^{(1)}-F_v^{(0)}\|_{Y}, 
\end{equation*}
\begin{equation*}
\|F_w^{(m)} - F_w^{(n)}\|_Y\leq \sum_{k=n}^{m-1} \left\|F_w^{(k+1)} - F_w^{(k)}\right\|_{Y}
\leq \sum_{k=n}^{m-1} \frac{1}{\gamma}\left\|F_v^{(k+1)} - F_v^{(k)}\right\|_{Y}\leq\sum_{k=n}^{m-1} \frac{1}{\gamma}\alpha(T)^k \|F_v^{(1)}-F_v^{(0)}\|_{Y}. 
\end{equation*}
Condition \eqref{condition_linear_estimate} guarantees that we have a contraction mapping
hence the sequence $\{(F_v^{(k)},F_w^{(k)})\}_k$ converges in $C([0,T];BC^0\times BC^0)$ and uniqueness of solution in Theorem \ref{lemma_local_existence} imply that $(F_v, F_w) \in C([0,\infty),\mathcal{B}\times \mathcal{B})$. 
To conclude the last part of the proof of the Lemma \ref{lemma_linear_estimate_error}, we need to prove the estimate \eqref{estimate_linear_error}, which can be obtained by bounding in a proper manner $\|F_v^{(1)}-F_v^{(0)}\|_Y$. The main difficulty is that we need a factor $\frac{1}{\omega_1}$ from
\begin{equation}\label{first_step_iteration_v}
F_v^{(1)} -F_v^{(0)}= \int_0^t e^{ (1-v_0^2)(t-\tau)} g_1(t-\tau) *  \varphi_3 d\tau, 
\end{equation}
where $\varphi_3$ is given by \eqref{def_varphi3}. The key for this estimate is the following lemma, which is why we require additional regularity for the Theorem \ref{thm_approximation_pas}.
\begin{lemma}[Oscillatory estimate for the linear non-homogeneous heat equation]\label{lemma_oscillatory_integral}
Let $\omega > 1$, let $d>0$ and let $f\in C^1(\R\times[0,T])$, then we have the following integral estimate
\begin{equation*}
\left\|\int_0^t \int_\R  \frac{e^{-\frac{|x-y|^2}{4 (t-\tau)} + (1-v_0^2)(t-\tau)}}{(4\pi (t-\tau))^{1/2}}  \frac{f(y,\tau)}{d^2+y^2} e^{i \omega \tau} dy d\tau \right\|_Y
\leq \frac{C}{\omega}\left(\|\partial_t f\|_Y + \|\partial_x f\|_Y + \|f\|_Y \right),
\end{equation*}
for some constant $C= C(v_0,d)>0$.
\end{lemma}
\begin{proof}
    We first observe that if $h(x,t) = f(x,t)/(d+x^2)$ then there exist a constant $C_d$ only depending on $d$ such that
    \begin{align*}
        \|\partial_t h\|_Y+\|\partial_x h\|_Y+ \|h\|_Y \leq C_d(\|\partial_t f\|_Y+\|\partial_x f\|_Y+ \|f\|_Y).
    \end{align*}
    Next, by considering the change of variables $s \to t-\tau$, $x-y \to 2\tau^{1/2} z$ we can write   
        \begin{align*}
            I(x,t)&= \int_0^t \int_\R  \frac{e^{-\frac{|x-y|^2}{4 (t-s)} + (1-v_0^2)(t-s)}}{(4\pi (t-s))^{1/2}}  h(y,s) e^{i \omega s} dy ds\\
            &=e^{i \omega t}\int_0^t e^{-(i\omega+(v_0^2-1))\tau}\frac{1}{\sqrt{\pi}}\int_\R  e^{-z^2} h(x-2\tau^{1/2}z,\tau) dz d\tau.
        \end{align*}
Letting $\psi(x,t) = \frac{1}{\sqrt{\pi}}\int_\R  e^{-z^2} h(x-2t^{1/2}z,t) dz $ we can verify that
\begin{align*}
    \|\psi(\cdot,t)\|_{L^\infty_x}\leq& ~ \|h(\cdot,t)\|_{L^\infty_x}\\
    \|\partial_t \psi (\cdot,t)\|_{L^\infty_x} =&  \bigg\| \frac{1}{\sqrt\pi}\int_\R  e^{-z^2} \Big(\partial_x h(x-2t^{1/2}z,t)(- t^{-1/2}z) + \partial_t h(x-2t^{1/2}z,t) \Big)   dz\bigg\|_{L^\infty_x}\\
    \leq &~ \frac{1}{\sqrt{t \pi}} \|\partial_x h(\cdot,t)\|_{L^\infty_x} + \|\partial_t h(\cdot,t)\|_{L^\infty_x}.
\end{align*}
    For each $x$ we integrate by parts in $\tau$, and use that $v_0^2-1>0$, to get that for any $(x,t)$ 
    \begin{align*}
        |I(x,t)|=&~\left| \int_0^t e^{-(i\omega+(v_0^2-1))\tau}\psi(x,t) d\tau \right| \\
        =&~\frac{1}{|i\omega+v_0^2-1|}  \left| e^{-(i\omega+(v_0^2-1))\tau}\psi(x,\tau) \Big|_{\tau=0}^t - \int_0^t e^{-(i\omega+(v_0^2-1))\tau} \partial_t \psi(x,\tau) d\tau \right| \\
        \leq &~ \frac{1}{\omega}\left(2\|h\|_Y+\Big(\|\partial_x h\|_Y+\|\partial_t h\|_Y\Big) \int_0^t e^{-(v_0^2-1)\tau}\left(\frac{1}{\sqrt{\tau\pi}}+1\right) d\tau\right)\\
        \leq &~ \frac{C_{v_0}}{\omega}\left(\|h\|_Y+\|\partial_x h\|_Y+\|\partial_t h\|_Y \right),
    \end{align*}  
    for some $C_{v_0}$ that only depends on $v_0^2-1$. We can finally conclude that $\|I(x,t)\|_Y\leq \frac{C}{\omega} \left(\|f\|_Y+\|\partial_x f\|_Y+\|\partial_t f\|_Y \right)$ for some $C=C(v_0,d)$ that only depends on $d$ and $v_0$.
\end{proof}
Now, we are in a position to end the proof of Lemma \ref{lemma_linear_estimate_error}. We can apply Lemma \ref{lemma_oscillatory_integral} to each term in \eqref{first_step_iteration_v}. Using the properties for $V$ in Proposition \ref{properties_solution_pas}. we conclude 
\begin{equation}\label{estimate_initial_Fv}
\|F_v^{(1)} -F_v^{(0)}\|_Y \leq \frac{C_1}{\omega_1}.
\end{equation}
Finally, using the convergence of the sequence $\{F_v^{(k)}\}$, the estimate \eqref{estimate_initial_Fv} and that $\alpha = \alpha(T) < 1$ we conclude
\begin{align*}
\|F_v-F_v^{(0)}\|_Y&\leq \|F_v- F_v^{(N+1)}\|_{Y} + \sum_{k=0}^{N}\left\|F_v^{(k+1)} - F_v^{(k)}\right\|_{Y}\\
&\leq \|F_v- F_v^{(N+1)}\|_{Y} + \sum_{k=0}^N \alpha^k \|F_v^{(1)}-F_v^{(0)}\|_{Y}\\
&\leq \|F_v- F_v^{(N+1)}\|_{Y} +\frac{\alpha}{1-\alpha} \frac{C}{\omega_1}.
\end{align*}
Therefore, by taking the limit as $N\to \infty$ we get the first estimate in \eqref{estimate_linear_error}. Analogously for $F_w$ we get
\begin{align*}
\|F_w-F_w^{(0)}\|_Y&\leq \|F_w- F_w^{(N+1)}\|_{Y} + \sum_{k=0}^{N}\left\|F_w^{(k+1)} - F_w^{(k)}\right\|_{Y}\\
&\leq \|F_w- F_w^{(N+1)}\|_{Y} + \sum_{k=0}^N \frac{1}{\gamma}\left\|F_v^{(k+1)} - F_v^{(k)}\right\|_{Y}\\
&\leq \|F_w- F_w^{(N+1)}\|_{Y} +\frac{1}{\gamma}\frac{ \alpha}{1-\alpha} \frac{C}{\omega_1}.
\end{align*}
Taking the limit as $N\to \infty$ we get the second part of \eqref{estimate_linear_error}. This concludes the proof of Lemma \ref{lemma_linear_estimate_error}.
\end{proof}

The next result establishes that the condition  \eqref{condition_linear_estimate}  in the previous lemma is not too restrictive.

\begin{lemma}\label{lemma_estimate_delta}
Let $\beta, \gamma$ and $\delta$ satisfy \eqref{condition_H1} and let $\varepsilon >0$. Let $A(x)$ and $B(x)$ be given by \eqref{decay_AxBx} and let $(V,W)$ the solution of \eqref{FHN_pde_pas}. Suppose that  
\begin{itemize}
\item[(i)] $\|V\|_Y\leq \min\left\{1,\frac{1}{\gamma}\frac{1}{1+2\max\{\sqrt{3},\beta\}}\right\}$,
\item[(ii)] $M = |a|/d_1^2+|b|/d_2^2 \leq \min\left\{ \frac{1}{\sqrt{\gamma}},\frac{1}{2\gamma(1+\max\{\sqrt{3},\beta\})}\right\}$.
\end{itemize}
Then, condition \eqref{condition_linear_estimate} in Lemma \ref{lemma_linear_estimate_error} is satisfied.
\end{lemma}
\begin{proof}
Using the bounds in $\frac{1}{v_0^2-1}$ and $|v_0|$ provided by Lemma \ref{properties_admissible_parameters} and our hypothesis we get 
\begin{align*}
\frac{\|v_0^2 - (v_0\!+\!V)^2\|_Y}{v_0^2-1} &= \frac{\|V (V\!+\!2v_0)\|_Y}{(v_0^2-1)}\leq \delta \gamma \|V\|_Y( \|V\|_Y\!+\!2\max\{\sqrt{3},\beta\}\}) \leq \delta,\\
\frac{M^2}{v_0^2-1} &\leq M^2 \delta\gamma \leq \delta ,\\
\frac{2M\|v_0+V \|_Y}{v_0^2-1} &\leq 2M (\|V\|_Y + \max\{\sqrt{3},\beta\})\delta\gamma\leq \delta,\\
\frac{1}{\gamma}\frac{1}{v_0^2-1}&\leq  \delta.
\end{align*}
We obtain that for all $0<\delta< 1/4$, condition \eqref{condition_linear_estimate} is satisfied. This concludes the proof of Lemma \ref{lemma_estimate_delta}.

\end{proof}

\subsection{Nonlinear estimate of the error}
In subsection \ref{subsection_linear_estimate_error}, we obtained an estimate for the linear part of the equation \eqref{full_error_system}. This subsection will explain how that estimate can be used to bound the solution of \eqref{full_error_system}. 
\begin{lemma}[Nonlinear estimate for the error equation]\label{nonlinear_estimate_error_pas}
Let $\varepsilon>0$, let $\gamma$, $\beta$ satisfy \eqref{condition_H1} and let $(v_0,w_0)$ be given by \eqref{defi_v0w0}.
For some $T>0$ let $(V,W)\in C([0,T];\mathcal{B}\times \mathcal{B})$ be a solution of \eqref{FHN_pde_pas}, let $\varphi_1, \varphi_2$ be given by \eqref{def_varphi1}, \eqref{def_varphi2}, and let $(F_v,F_w)\in C([0,T];\mathcal{B}\times \mathcal{B})$ be the corresponding solution of \eqref{equation_linear_estimate_error}.

Given $\mu >0$ there exist constants $C_1(\mu)$, $C_2(\mu)$, $C_3(\mu)$ $>0$ such that 
if 
\begin{enumerate}
\item $\|V\|_Y \leq C_1(\mu)$,
\item $\|\varphi_1\|_Y\leq C_2(\mu)$,
\item $\|\varphi_2\|_Y \leq |v_0| + C_1(\mu) + 1$,
\item $|F_v(x,t)| \leq C_3(\mu)$, $|F_w(x,t)|\leq \mu/2$, for $x\in \R$, $0\leq t\leq T$,
\end{enumerate}
then, there exists a unique solution $(E_v,E_w)\in C([0,T];\mathcal{B}\times \mathcal{B} )$ of \eqref{full_error_system} that also satisfies the estimate
\begin{equation}\label{thm_nonlinear_conclusion}
|E_v(x,t)| \leq \mu,\quad |E_w(x,t)| \leq \mu, \quad \forall x\in \R,\quad 0\leq t \leq T.
\end{equation}

\end{lemma}
\begin{proof}
First, assume $C_3(\mu)\leq \mu/2$, so we have
\begin{equation}\label{estimates_nonlinear_error_2}
|F_v(x,t)| \leq \mu/2, \quad |F_w(x,t)|\leq \mu/2,\quad  \forall x\in \R,\quad  0\leq t\leq T.
\end{equation}
Since we want to use \eqref{equation_linear_estimate_error} to study \eqref{full_error_system} it makes sense to consider the trajectories
\begin{align*}
R_v(x,t) &= E_v(x,t) - F_v(x,t), \\
R_w(x,t) &= E_w(x,t) - F_w(x,t).
\end{align*}
The equation for $(R_v,R_w)$ is given by
\begin{align}\label{eqR_vR_W}
\left\{
\begin{array}{rl}
\partial_t R_v - \partial_x^2 R_v &= (1-(v_0+V)^2 + \varphi_1) R_v -\frac{1}{3}R_v^{3} -R_w -R_v^2 F_v \\
& \hspace{1cm}-R_v F_v^2-\frac{1}{3}F_v^3 + \varphi_2 (R_v + F_v)^2,\\
\partial_t R_w-\rho\partial_x^2 R_w &= \varepsilon(R_v- \gamma R_w),\\
R_v(x,0)=0,  &R_w(x,0) = 0,
\end{array}
\right.
\end{align}
where $\varphi_1$ and $\varphi_2$ are given by \eqref{def_varphi1} and \eqref{def_varphi2}, respectively. We must establish that the nonlinear error system \eqref{eqR_vR_W} admits small contracting rectangles. Consider the vector field
\begin{align}
X((R_v, R_w),x,t) &= \binom{X_1((R_v, R_w),x,t) }{X_2((R_v, R_w),x,t) }\label{vector_field_error_eqn}\\
X_1((R_v, R_w),x,t) &= 
(1-(v_0+V)^2 + \varphi_1) R_v -\frac{R_v^{3}}{3} -R_w -R_v^2 F_v \notag\\
&\hspace{1cm}-R_v F_v^2 -\frac{F_v^3}{3} + \varphi_2 (R_v + F_v)^2 \notag\\
X_2((R_v, R_w),x,t) &= \varepsilon(R_v- \gamma R_w) \notag
\end{align}
which belongs to $V(\mathcal{B})$ given by Definition \ref{defi_space_vB} for the spaces $\mathcal{B}$ in Proposition \ref{prop_hip_local_existence_PAS}. As in the proof of Lemma \ref{lemma_invariant_rectangles}, we look for contracting rectangles by checking each face individually.
\begin{enumerate}
\item Top/Bottom face. We get the condition $S > \frac{1}{\gamma} L$. Take $S= \frac{1+\epsilon}{\gamma}L$ where $\epsilon >0$ is chosen so that
\begin{equation}\label{condition_choice_epsilon}
(v_0^2 -1 - \frac{1+\epsilon}{\gamma})>0.
\end{equation}
This can be done because $v_0^2 -1 - \frac{1}{\gamma} >0$ as implied by \eqref{condition_v0}. 

\item Left face. For  $R_v= - L$, $R_w\in [- S,S]$ we get the following
\begin{align*}
I&= (-1,0)\cdot X \\
&=-(1-(v_0+V)^2 -\varphi_1) R_v - \varphi_2 R_v^2 +\frac{R_v^{3}}{3} +R_w \\
 &\hspace{1cm}- F_v\left( -R_v^2 -R_v F_v -\frac{F_v^2}{3} + 2\varphi_2 R_v  +  \varphi_2 F_v \right)\\
&\leq (1-(v_0+V)^2 -\varphi_1) L - \varphi_2 L^2 -\frac{L^{3}}{3} +S \\
&\hspace{1cm}- F_v\left( -L^2 +L F_v -\frac{F_v^2}{3} - 2\varphi_2 L  +  \varphi_2 F_v \right)\\
&\leq -((v_0+V)^2 - 1 - \|\varphi_1\|_Y)L+ \|\varphi_2\|_Y L^2 - L^3/3 + S\\
&\hspace{1cm}+\|F_v\|_Y(L^2 + L \|F_v\|_Y + \frac{1}{3}\|F_v\|_Y^2 + 2 \|\varphi_2\|_Y L + \|\varphi_2\|_Y \|F_v\|_Y).
\end{align*}
Using $S= \frac{1+\epsilon}{\gamma}L$ we obtain the following polynomial in $L$
\begin{align*}
I &= (-1,0)\cdot X\\
&\leq -\left((v_0+V)^2-1-\frac{1+\varepsilon}{\gamma}-\|\varphi_1\|_Y-\|F_v\|_Y^2-2\|\varphi_2\|_Y\|F_v\|_Y\right)L\\
& \hspace{1cm}+(\|\varphi_2\|_Y +\|F_v\|_Y)L^2- \frac{1}{3}L^3 + \left(\frac{1}{3} \|F_v\|^3 +\|\varphi_2\|_Y \|F_v\|_Y^2\right)\\
&= p_0- p_1 L +p_2 L^2 - \frac{1}{3}L^3. 
\end{align*}
Hence, to obtain $(-1,0)\cdot X <0$, it is enough to find conditions on the coefficients $p_0$, $p_1$, $p_2$ so that $p_0- p_1 L +p_2 L^2 = 0$ for an $\hat L>0$. First, because of \eqref{condition_choice_epsilon} we know that $(v_0^2 -1 - \frac{1+\epsilon}{\gamma}) >0$. Second, we have the bounds $p_0 \leq P$, $p_1 \geq Q$, $p_2 \leq R$ for
\begin{enumerate}
\item $P= \frac{1}{3} C_3(\mu)^3 +(|v_0|+C_1(\mu) + 1) C_3(\mu)^2$,
\item $Q= (v_0^2 -1 - \frac{1+\epsilon}{\gamma}) +  (-(2|v_0|+C_1(\mu))C_1(\mu)-C_2(\mu)-C_3(\mu)^2-2(|v_0|+C_1(\mu) + 1) C_3(\mu))$,
\item $R =(|v_0|+C_1(\mu) + 1) +C_3(\mu)$.
\end{enumerate}
Third, it is possible to choose the constants $C_1(\mu)$, $C_2(\mu)$, $C_3(\mu) >0$ small enough such that
\begin{enumerate}
\item $Q>0$,
\item $4PR <Q^2$,
\item $0<\hat L := \frac{Q -\sqrt{Q^2 - 4PR}}{2R} < \min \{\frac{\mu}{2}, \frac{\gamma \mu}{2(1+\epsilon)}\}$.
\end{enumerate}
Lastly, using the bound
$p_0-p_1{\hat L} + p_2{\hat L}^2 \leq P - Q{\hat L} + R{\hat L}^2 =0$ we get for $L=\hat L$ that $(-1,0)\cdot X \leq -{\hat L}^3/3< 0$, and 
\begin{equation}\label{condition_choice_L}
0<\hat{L}<\min\left\{ \frac{\mu}{2}, \frac{\gamma \mu}{2(1+\epsilon)}\right\}.
\end{equation}
\item Right face. Analogously to the left face, 
for $R_v= L$, $R_w\in [- S,S]$ we get the following
\begin{align*}
J&=(1,0)\cdot X \\
&\leq (1-(v_0+V)^2 -\varphi_1) L + \varphi_2 L^2 -\frac{L^{3}}{3} +S \\
&\hspace{1cm}+ F_v\left( -L^2 -L F_v -\frac{F_v^2}{3} + 2\varphi_2 L  +  \varphi_2 F_v \right)\\
&\leq -\left((v_0+V)^2 - 1 - \frac{1+\epsilon}{\gamma} - \|\varphi_1\|_Y- \|F_v\|^2_Y - 2 \|\varphi_2\|_Y \|F_v\|_Y\right)L\\
&\hspace{1cm}+ (\|\varphi_2\|_Y+\|F_v\|_Y) L^2 \\
&\hspace{1cm} - L^3/3 +\left(\frac{1}{3}\|F_v\|_Y^3 + \|\varphi_2\|_Y \|F_v\|_Y^2\right).
\end{align*}
This is the same condition we obtained for the left face.
\end{enumerate}

Finally, because $(1+\hat{\epsilon})(0,0)\in R_{L,S}$ for any rectangle, we can apply Theorem \ref{thm_well_posedness} with $(L,S) = (\hat{L},\frac{1+\epsilon}{\gamma} \hat{L})$, since we showed that such rectangle is contracting under the vector field \eqref{vector_field_error_eqn}. Therefore, there exists a unique solution $(R_v, R_w)\in C([0,T];\mathcal{B} \times \mathcal{B})$ of \eqref{eqR_vR_W} which also satisfies
$$(R_v, R_w) \in R_{L,S}, \quad \forall x\in \R,\quad 0\leq t \leq T.$$ 
Additionally, because of \eqref{condition_choice_L} we know that 
\begin{equation}\label{estimates_nonlinear_error_1}
|R_v(x,t)|\leq \mu/2,\quad |R_w(x,t)|\leq \mu/2 , \quad \forall x\in \R,\quad 0\leq t \leq T.
\end{equation}
We conclude that there exists a unique solution $(E_v, E_w) = (F_v+R_v,F_w+R_w)\in C([0,T];\mathcal{B}\times \mathcal{B})$ of \eqref{full_error_system}. Combining \eqref{estimates_nonlinear_error_2} and \eqref{estimates_nonlinear_error_1}, we obtain the estimate \eqref{thm_nonlinear_conclusion} and conclude the proof of Lemma \ref{nonlinear_estimate_error_pas}.
\end{proof}

\subsection{Proof of Theorem \ref{thm_approximation_pas}}
We finally have all the ingredients to complete the proof of Theorem \ref{thm_approximation_pas}. Given $\mu >0$ choose $C_1(\mu)$, $C_2(\mu)$, $C_3(\mu) >0$ as in Lemma \ref{nonlinear_estimate_error_pas}. Fix $T>0$; to use the above results, we will need $\|V\|_Y$ and $|a|/d_1^2+|b|/d_2^2$ to be small enough. To make this more precise, assume we had $\|V\|_Y<M_1$ and $|a|/d_1^2+|b|/d_2^2<N_1$. If $M_1,N_1$ were small, as indicated in Lemma \ref{lemma_estimate_delta}, then condition \eqref{condition_linear_estimate} would be fulfilled and Lemma \ref{lemma_linear_estimate_error} would imply
\begin{equation}\label{teo3_prop_FvFw}
|F_v(x,t)| \leq C_3(\mu),\quad |F_w(x,t)|\leq \mu/2, \quad \forall x\in \R,~ t\in[0,T],~ \omega_1 \geq \omega_0,
\end{equation}
for some $\omega_0>0$ that depends only on $\mu$. And if $\|V\|_Y<M_1$ and $|a|/d_1^2+|b|/d_2^2<N_1$ we would also have the bounds
\begin{enumerate}
    \item[$\bullet$] $|(v_0^2-(v_0+V)^2| \leq M_1 (2|v_0|+M_1)$,
    \item[$\bullet$] $|J_0|\leq N_1$,
    \item[$\bullet$] $|\varphi_1| =|J_0^2+2(v_0+V)J_0|\leq N_1(N_1+2(|v_0|+M_1) )$,
    \item[$\bullet$] $|\varphi_2| = |v_0+V+J_0|\leq |v_0|+M_1+N_1$.
\end{enumerate}
We will choose $M_1,N_1$ small enough such that whenever $\|V\|_Y<M_1$ and $|a|/d_1^2+|b|/d_2^2<N_1$ then we also have $\|V\|_Y\leq C_1(\mu), |\varphi_1|\leq C_2(\mu)$ and $|\varphi_2|\leq |v_0|+C_1(\mu)+1$. We will also require $N_1<N$ for the $N$ prescribed by Theorem \ref{thm_estimates_pas}.
With all these choices taken care of, we can put all the previous results together.

Given an open neighborhood ${\mathcal{O}}$ of $(0,0)$ we denote  $\tilde{\mathcal{O}} = {\mathcal{O}} \cap (-M_1,M_1)\times (-1,1)$ and let $R$ be the rectangle given by Theorem \ref{thm_estimates_pas} corresponding to $\tilde{\mathcal{O}}$. Set the rectangle in Theorem \ref{thm_approximation_pas} to be this rectangle $R$ and set $N$ in Theorem \ref{thm_approximation_pas} to be $N_1$. Next, by assumption, the initial data $(f_0-v_0,g_0-w_0)\in \mathcal{B}\times \mathcal{B}$ satisfies
$$(1+\hat{\epsilon})(f_0(x)-v_0,g_0(x)-w_0)\in R,\quad \forall x\in\R$$
and from Theorem \ref{thm_estimates_pas} we know that there exists a unique solution $(V,W)\in C([0,\infty);\mathcal{B}\times \mathcal{B})$ of the initial value problem \eqref{FHN_pde_pas} with initial data $(f_0-v_0,g_0-w_0)$, which also satisfies
$$(V(x,t),W(x,t))\in R, \quad \forall x\in\R,\quad t\geq0.$$
In particular, by construction of $R$, we have that $|V(x,t)|\leq M_1, \forall x\in\R, t>0$. Fix $T>0$, from the choices of $M_1,N_1$, the initial value problem \eqref{equation_linear_estimate_error} has a unique solution $(F_v, F_w)\in C([0,T];\mathcal{B}\times \mathcal{B})$ by Lemma \ref{lemma_linear_estimate_error}. We also verify the hypotheses of Lemma \ref{nonlinear_estimate_error_pas}, so we conclude there exists a unique solution $(E_v, E_w)\in C([0,T];\mathcal{B}\times \mathcal{B})$ of \eqref{full_error_system} which satisfies 
\begin{equation}\label{final_estimate_error}
|E_v(x,t)| \leq \mu,\quad |E_w(x,t)| \leq \mu, \quad \forall x\in \R, t \in[0,T].
\end{equation}
Finally, using the definition of the equation for the approximation error in Proposition \ref{prop_eqn_approximation_error}, we have the decomposition
$$v = V + E_v + J_0, \quad w = W + E_w,\quad x\in\R,t\in[0,T],$$
and we conclude that there exits a unique solution $(v,w)\in C([0,T];\mathcal{B}\times \mathcal{B})$ of the initial value problem \eqref{FHN_pde} with initial data $(\bar{f}_0,\bar{g}_0) = (f_0 - v_0,g_0-w_0)$ which satisfies \eqref{final_estimate_error}. Returning to the original variables, there exists a unique solution $(f,g)$ of \eqref{FHN_pde_non_centered} with initial data $(f_0,g_0)$ that satisfies $(f -v_0 ,g-w_0 )\in C([0,T];\mathcal{B}\times \mathcal{B})$, and that additionally satisfies the estimate
\begin{equation*}
|f-v_0-V-J_0| \leq \mu,\quad |g-w_0-W| \leq \mu, \quad \forall x\in \R, t\in[0,T].
\end{equation*}
Since $T>0$ was arbitrary, the proof of Theorem \ref{thm_approximation_pas} is complete. \qed

\appendix
\section{Appendix: Technical proofs.}\label{section_appendix}
\subsection{Derivation of the Partially Averaged System}\label{appendix_derivation}
We look for solutions $(v,w)$ of \eqref{FHN_pde} of the form 
\begin{align*}
    v=\bar{v}+A(x)\sin(\omega_1 t)+B(x)\sin(\omega_2 t)+ O\left(\frac{1}{\omega_1}\right),\quad
    w=\bar{w}+O\left(\frac{1}{\omega_1}\right),
\end{align*}
where $(\bar{v},\bar{w})$ is the slow varying part of $(v,w)$, $A(x)$ and $B(x)$ are given by \eqref{decay_AxBx} and $\omega_1$, $\omega_2$ satisfy \eqref{interferential}. In order to obtain equations for $(\bar{v},\bar{w})$, we ignore the error terms $O\left(\frac{1}{\omega_1}\right)$ and we substitute $v=\bar{v}+A(x)\sin(\omega_1 t)+B(x)\sin(\omega_2 t)$, and $w=\bar{w}$ into \eqref{FHN_pde}. Using the notation 
$$J_0(x,t)=A(x)\sin(\omega_1 t)+B(x)\sin(\omega_2 t)$$
 we can write the following equations for $(\bar v,\bar w)$
\begin{align}
    \partial_t \bar{v}-\partial_{x}^2 \bar{v}&=(1-v_0^2)\bar{v} - v_0 \bar{v}^2-\frac{1}{3}\bar{v}^3-\bar{w} -\bar{v}X_1(x,t) -Z(\bar{v},x,t)-v_0 J_0^2 \notag \\
    &=\psi_1(\bar{v},\bar{w},x,t),\label{pre_avg}\\
    \partial_t\bar{w}-\rho\partial_{x}^{2}\bar{w} &=\varepsilon(\bar{v}-\gamma \bar{w}) +\varepsilon J_0=
 \psi_2(\bar{v},\bar{w},x,t), \notag
\end{align}
where $X_1(x,t)$ and $Z(x,t)$ are given by 
\begin{multline*}
    X_1(x,t)= A(x)^2  \sin^2(\omega_1 t)+ B(x)^2  \sin^2(\omega_2 t)\\
    +2 A(x) B(x)  \sin(\omega_1 t) \sin(\omega_2 t) +2 v_0 J_0 ,
\end{multline*}
and 
\begin{align*}
Z(\bar{v},x,t) &= - A(x) \sin(\omega_1 t) -B(x) \sin(\omega_2 t)\\
&\hspace{1cm}+ \frac{A(x)^3}{3} \sin^3(\omega_1 t) + \frac{B(x)^3}{3} \sin^3(\omega_2 t) \\
&\hspace{1cm}+ A(x)\bar{v}^2 \sin(\omega_1 t) +B(x)\bar{v}^2 \sin(\omega_2 t) \\
&\hspace{1cm}+  A(x)^2 B(x)\sin^2(\omega_1 t) \sin(\omega_2 t)\\
&\hspace{1cm}+  A(x) B(x)^2 \sin(\omega_1 t) \sin^2(\omega_2 t)- \partial_x^2 J_0(x,t)  +v_0^2 J_0.
\end{align*}
We assume that $(\bar{v},\! \bar{w})$ is slowly varying in comparison with the high-frequency term $J_0$. Thus, it makes sense to consider the following averaging on the right-hand side, where we assume $V$ and $W$ to be constants 
\begin{align*}
   \left| \frac{\omega_1}{2\pi}\int_{t-\frac{\pi}{\omega_1}}^{t+\frac{\pi}{\omega_1}}
    Z(V,x,s)ds \right| \leq C(1\!+\!V^2)(\vert A(x)\vert \!+\! \vert B(x)\vert \!+\! \vert A(x)\vert ^3 \!+\! \vert B(x)\vert ^3)\frac{\eta}{\omega_1}.
\end{align*}
In $X_1(x,t)$ we use the identities $\sin^2 (\alpha) = \frac{1}{2} - \frac{1}{2}\cos (2\alpha)$ and $\sin(\omega_1 t) \sin(\omega_2 t)$ $= \frac{1}{2}\left(\cos(\omega_1-\omega_2)t - \cos(\omega_1+\omega_2)t\right)$ to get the following decomposition
\begin{equation*}
X_1(x,t) = \frac{A(x)^2}{2} + \frac{B(x)^2}{2} + A(x) B(x) \cos(\omega_1 - \omega_2) t + X_2(x,t),
\end{equation*}
where 
$X_2(x,t) =- \frac{A(x)^2}{2} \cos(2\omega_1 t) - \frac{B(x)^2}{2} \cos(2\omega_2 t) - A(x) B(x) \cos((\omega_1 + \omega_2)t) + 2 v_0 J_0 $. Taking the same averaging as before and using $\omega_1 - \omega_2 =  -\eta$ we get
\begin{multline*}
\frac{\omega_1}{2\pi}\int_{t-\pi/\omega_1}^{t+\pi/\omega_1}\left(\frac{A(x)^2}{2} + \frac{B(x)^2}{2} + A(x)B(x) \cos(\eta s) \right)ds\\
= \frac{A(x)^2}{2} + \frac{B(x)^2}{2} + A(x)B(x) \cos(\eta t) + O\left(\frac{\vert A(x) B(x) \eta\vert }{\omega_1}\right),
\end{multline*}
\begin{equation*}
\left\vert \frac{\omega_1}{2\pi}\int_{t-\pi/\omega_1}^{t+\pi/\omega_1} X_2(x,s) ds\right\vert  \leq C (|A(x)|+|B(x)|+|A(x)|^2+|B(x)|^2 ) \frac{\eta}{\omega_1}.
\end{equation*}
Finally, for the term $v_0J_0^2$ in \eqref{pre_avg}, we get that
\begin{equation*}
\frac{\omega_1}{2\pi}\int_{t-\pi/\omega_1}^{t+\pi/\omega_1}v_0J_0^2(s) ds
= v_0\left(\frac{A(x)^2}{2} + \frac{B(x)^2}{2} + A(x)B(x) \cos(\eta t)\right) 
+ O\left(\frac{ \left(|A(x)|^2 +|B(x)|^2\right) \eta }{\omega_1}\right).
\end{equation*}
Putting all together, we get from the averaging on $[t-\pi/\omega_1,t+\pi/\omega_1]$
\begin{align*}
I_1 &= \frac{\omega_1}{2\pi}\int_{t-\pi/\omega_1}^{t+\pi/\omega_1} \psi_1(V,W,x,s)ds \\
&= V\left(1- v_0^2-\frac{A(x)^2}{2}-\frac{B(x)^2}{2}-A(x)B(x)\cos(\eta t)\right)\\
&\hspace{1cm}-v_0 V^2- V^3/3 - W -v_0\left(\frac{A(x)^2}{2}+\frac{B(x)^2}{2}+A(x)B(x)\cos(\eta t)\right)\\
&\hspace{1cm}+O\left(( V^2 +1) (\vert A(x)\vert +\vert B(x)\vert +\vert A(x)\vert ^3 + \vert B(x)\vert ^3) \eta/\omega_1\right),\\
I_2 &=\frac{\omega_1}{2\pi}\int_{t-\pi/\omega_1}^{t+\pi/\omega_1} \psi_2(V,W,x,s)ds \\
&= \varepsilon(V-\gamma W) +O(\vert B(x)\vert \eta/\omega_1).
\end{align*}
Therefore, by ignoring all the $O(1/\omega_1)$ terms, we get the Partially Averaged System given by  \eqref{FHN_pde_pas}, i.e.,
\begin{equation*}
\left\{
\begin{array}{rl}
\partial_t V-\partial_{x}^2V &= \left(1-v_0^2-\frac{A(x)^2}{2} - \frac{B(x)^2}{2} -A(x) B(x) \cos(\eta t) \right)V\\
&\hspace{1.5cm}-v_0V^2 - \frac{1}{3}V^3 - W\\
&\hspace{1.5cm}-\left(\frac{A(x)^2}{2}+\frac{B(x)^2}{2}+A(x)B(x)\cos(\eta t)\right)v_0,\\
\partial_t W- \rho\partial_{x}^2 W &= \varepsilon\left( V- \gamma W \right),\\
V(x,0)  = \bar{f_0}(x),\; &W(x,0) =\bar{g_0}(x) .
\end{array}
\right.
\end{equation*}

\subsection{Proof of Lemma \ref{properties_admissible_parameters}}\label{app_properties_admissible_parameters}
Solving equation \eqref{defi_v0w0} is equivalent to solving the following cubic equation for $v_0$
\begin{equation}\label{defi_polinomial_h}
h(v) = v^3 - 3(1-1/\gamma)v + 3\beta/\gamma  = 0.
\end{equation}
For such depressed cubic, there always is a solution, and the condition for the uniqueness of the real root can be written as
\begin{equation}\label{first_condition_paramaters}
\frac{(1-\gamma)^3}{\gamma^3} +\frac 9 4 \frac{\beta^2}{\gamma^2}  > 0,
\end{equation}
so if $\beta,\gamma$ are admissible then \ref{lemma_admis_1} in Lemma \ref{properties_admissible_parameters} holds. Next, we verify that condition \eqref{condition_H1} is non-empty. Since all parameters are positive and $\frac{\gamma-1}{\gamma}\leq 1$, condition \eqref{first_condition_paramaters} always holds if we require $\frac{9}{4}\frac{\beta^2}{\gamma^2} \geq 1$ or simply $\frac{\beta}{\gamma}\geq \frac{2}{3}$. Once $\beta/\gamma\geq 2/3$ the second condition in \eqref{condition_H1} will be satisfied if the parameters verify that 
$$\left(1+\frac{1}{\delta\gamma}\right)^{1/2}\left( 2-\frac{3}{\gamma}-\frac{1}{\delta\gamma}\right) + 2 >0,$$
for which it is enough that $2-\frac{3}{\gamma}-\frac{1}{\delta\gamma} \geq 0$, or simply $\gamma \geq \frac{3\delta +1}{2\delta}$. We conclude that the set of admissible parameters is non-empty, and for any $\delta>0$, it includes the set
\begin{equation*}\left\{(\beta,\gamma): \gamma \geq \frac{2\delta +1}{\delta}, \beta \geq \frac{2}{3}\gamma\right\}.\end{equation*} To get the estimates in \eqref{condition_v0} of \ref{lemma_admis_2}, Lemma \ref{properties_admissible_parameters}, we use that $v_0$ is the unique solution of the cubic equation \eqref{defi_polinomial_h} and that the dominating coefficient in $h(v)$ is positive, therefore $a<v_0<b$ if and only if $h(a)<0<h(b)$. Let us compute
\begin{align*}
h\left(-\left(1+\frac{1}{\delta\gamma}\right)^{1/2}\right) &= -\left(1+\frac{1}{\delta\gamma}\right)^{3/2} + 3\left(1-\frac{1}{\gamma}\right)\left(1+\frac{1}{\delta\gamma}\right)^{1/2} + 3\frac{\beta}{\gamma} \notag,\\
&= \left(1+\frac{1}{\delta\gamma}\right)^{1/2}\left( 2-\frac{3+1/\delta}{\gamma}\right) + 3\frac{\beta}{\gamma},
\end{align*}
and this last quantity is positive because of \eqref{condition_H1}, hence $\left(- \sqrt{1+\frac{1}{\delta\gamma}}\right)>v_0$, giving the upper bound in \eqref{condition_v0}. To establish the lower bound in \eqref{condition_v0} we again evaluate $h$,
\begin{equation*}
h(-\sqrt{3})= -\frac{3}{\gamma} (\sqrt{3} -\beta),\quad h(-\beta)= -\beta(\beta-\sqrt{3})(\beta+\sqrt{3}).
\end{equation*}
For $\beta < \sqrt{3}$ we have $h(-\sqrt{3})<0$ and for $\beta \geq \sqrt{3}$ we get $h(-\beta)\leq 0$, hence $v_0\geq\min\{-\beta,-\sqrt{3}\}$. This concludes the proof of Lemma \ref{properties_admissible_parameters}. \qed

\subsection{Proof of Proposition \ref{prop_hip_local_existence_PAS}}\label{proof_prop_hip_local_existence_PAS}
We need the following interpolation result for the proof of Proposition \ref{prop_hip_local_existence_PAS}.
\begin{lemma}[Interpolation Lemma]\label{interpolation_lemma}
Let $k\geq 1$ be a positive integer and $p \geq 1$ be real. Then, we have the following results where $C>0$ denotes a generic constant.
\begin{itemize}
\item[(i)]  If $f\in W^{k,p}(\R)$, $j\in \N_0$, $j \leq k$ and $j/k \leq \theta \leq 1$, then
$$\|D^{j} f\|_{L^{p k/j}}\leq C \|f\|_{W^{k,p}}^{\theta} \|f\|_{L^\infty}^{1-\theta}.$$
\item[(ii)] If $f,g\in W^{k,p}(\R)$,  then
\begin{equation}\label{product_wkp}
\|f g\|_{W^{k,p}}\leq C (\|f\|_{W^{k,p}}+\|g\|_{W^{k,p}})(\|f\|_{L^\infty}+\|g\|_{L^\infty}).
\end{equation}
\item[(iii)] If $f,g\in W^{k,p}(\R)$,  and $kp>1$  then 
\begin{equation}\label{cauchy_wkp}
\|f g\|_{W^{k,p}} \leq C \|f\|_{W^{k,p}} \|g\|_{W^{k,p}}.
\end{equation}
\end{itemize}
\end{lemma}
\begin{proof}
Let us start with the item (i). The case $j = k$ is immediate. For $j<k$, using Gagliardo-Niremberg interpolation, we know that
$$\|D^j f\|_{L^{\hat{p}}} \leq C \|D^k f\|_{L^p}^\theta \|f\|_{L^{q}}^{1-\theta}$$
where $\frac{1}{\hat{p}} = \frac{\theta}{p} + k(\frac{j}{k} - \theta)  + \frac{1-\theta}{q}$, $j/k \leq \theta \leq 1$. We apply this inequality with $\theta = j/k$ $q = \infty$ to obtain 
\begin{equation*}
\|D^j f\|_{L^{p k/j}} \leq C \|D^k f\|_{L^p}^{j/k} \|f\|_{L^{\infty}}^{(1-j/k)}.
\end{equation*}
Next, applying the inequality with $q = p$ and $\theta = j/k+\frac{\hat{p}-p}{p\hat{p}}$, $\hat{p} = p k/j$ we obtain
\begin{equation*}
\|D^j f\|_{L^{p k/j}} \leq C \|D^k f\|_{L^p}^{j/k+\frac{\hat{p}-p}{p\hat{p}}} \|f\|_{L^{p}}^{(1-j/k -\frac{\hat{p}-p}{p\hat{p}})},
\end{equation*}
and taking the geometric average of the estimates
$$\|D^j f\|_{L^{pk/j}} \leq C \|D^k f\|_{L^p}^{j/k+\delta \frac{\hat{p}-p}{p\hat{p}}} \|f\|_{L^{\infty}}^{(1-\delta)(1-j/k)} \|f\|_{L^{p}}^{\delta(1-j/k-\frac{\hat{p}-p}{p\hat{p}})}.$$
Using that $\|D^k f\|_{L^p}\leq \|f\|_{W^{k,p}}$, and $\|f\|_{L^p}\leq \|f\|_{W^{k,p}}$ this gives us item (i). For item (ii), we use the following
\begin{equation*}
\|f g\|_{L^{p}}\leq \|f\|_{L^{2p}} \|g\|_{L^{2p}} \leq \|f\|_{L^p}^{1/2} \|f\|_{L^\infty}^{1/2} \|g\|_{L^p}^{1/2} \|g\|_{L^\infty}^{1/2}
\leq C(\|f\|_{L^p}+\|g\|_{L^p})(\|f\|_{L^\infty}+\|g\|_{L^\infty}),
\end{equation*}
and for the derivative, we use estimate (i)
\begin{align*}
\|D^k(f g)\|_{L^{p}}&\leq \sum_{j=0}^{k}\binom{k}{j}\|D^j f D^{k-j}g\|_{L^{p}}\\
&\leq C \sum_{j=0}^{k}\binom{k}{j} \|D^j f\|_{L^{p k/j}} \|D^{k-j} g\|_{L^{p k/(k-j)}}\\
&\leq C \sum_{j=0}^{k}\binom{k}{j} \| f\|_{W^{k,p}}^{j/k} \|f\|_{L^\infty}^{1-j/k} \| g\|_{W^{k,p}}^{1-j/k} \|g\|_{L^\infty}^{j/k}\\
&\leq C (\|f\|_{W^{k,p}}+\|g\|_{W^{k,p}})(\|f\|_{L^\infty}+\|g\|_{L^\infty}).
\end{align*}
Combining the two inequalities, we obtain \eqref{product_wkp}. 
For item (iii), since $kp>1$, we can apply Sobolev Embedding to conclude that $\|f\|_{L^\infty} \leq \|f\|_{W^{k,p}}$, which applied to \eqref{product_wkp} gives us \eqref{cauchy_wkp}. This concludes the proof of Lemma \ref{interpolation_lemma}.
\end{proof}

For the proof of Proposition \ref{prop_hip_local_existence_PAS}, we must verify the properties listed in Definition \ref{space_rauch_smoller}. Property 1 is a consequence of Sobolev embedding in the case $\mathcal{B}=W^{k,p}(\R)$ and for $\mathcal{B}=BC^0\cap L^p(\R)$ it is a consequence of the definition of the norm $\|f\|_{BC^0\cap L^p(\R)} = \|f\|_{L^\infty} + \|f\|_{L^p} \geq \|f\|_{L^\infty}$. Properties 2 and 3 are immediate from the definition of the spaces in both cases.  

For Property 4 in Definition \ref{space_rauch_smoller}, we use that the space $V(\mathcal{B})$ given by Definition \ref{defi_space_vB} can be constructed recursively increasing the degree of the polynomial the following way: 
\begin{itemize}
\item[i)] Base step: Every $f\in L^\infty([0,\infty);\mathcal{B} \times \mathcal{B})$ belongs to $V(\mathcal{B})$. The functions $P_1((U,V),x,t) = U$, $P_2((U,V),x,t) = V$ belong to $V(\mathcal{B})$
\item[ii)] Recursive step:
\begin{itemize}
\item[a)] If $A \in V(\mathcal{B})$ and $f \in L^{\infty}([0,\infty))$, then $f A \in V(\mathcal{B})$
\item[b)] If $A_1 \in V(\mathcal{B})$ and $A_2\in V(\mathcal{B})$ then $A_1 A_2 \in V(\mathcal{B})$ (componentwise multiplication).
\item[c)] If $A_1 \in V(\mathcal{B})$ and $A_2\in V(\mathcal{B})$ then $A_1+A_2 \in V(\mathcal{B})$.
\end{itemize}
\item[iii)] The set $V(\mathcal{B})$ only contains elements obtained from i) or ii).
\end{itemize}
Next, we argue using structural induction that  all the elements in $A\in V(\mathcal{B})$ satisfy the property 
\begin{equation}\label{property}
A \text{ satisfies } \eqref{condition_non_linearity}, \eqref{condition_non_linearity2}\text{ and } \eqref{condition_non_linearity3}.
\end{equation}

Let us take a look at each family of spaces in the statement.\\
{\bf Case $\mathcal{B} = W^{k,p}(\R)$ with $kp>1$}\\
For elements given by the base step, the property \eqref{property} holds trivially. We proceed to the recursive step. Every time we write $A\in V(\mathcal{B})$, we mean a vector-valued function $A((V, W),x,t)$, but to simplify the notation, we only write $A(V, W)$.
\begin{enumerate}
\item[a)] Let $A\! \in\! V(W^{k,p}(\R))$ satisfying property \eqref{property} and let $f(t)\!\in\! L_t^{\infty}([0,\infty))$. Let us verify \eqref{condition_non_linearity}. Let 
\begin{equation*}
\|(V_1,W_1)\|_{W^{k,p}\times W^{k,p}} \leq M, \quad  \|(V_2,W_2)\|_{W^{k,p}\times W^{k,p}}\leq M
\end{equation*}
then
\begin{align*}
I_1 &= \|f(t) A(V_1,W_1,x,t) - f(t) A(V_2,W_2,x,t)\|_{W^{k,p}\times W^{k,p}}\\
&\leq \|f\|_{L_t^\infty([0,\infty))} \|A(V_1,W_1) - A(V_2,W_2)\|_{W^{k,p}\times W^{k,p}}\\
&\leq C \|f\|_{L_t^\infty([0,\infty))}\|(V_1-V_2,W_1-W_2)\|_{W^{k,p}\times W^{k,p}}
\end{align*}
because $A$ satisfies \eqref{condition_non_linearity}. Properties \eqref{condition_non_linearity2} and \eqref{condition_non_linearity3} are analogous. We conclude that $f A$ satisfies \eqref{property}.
\item[b)] Let $A \in V(W^{k,p}(\R))$ and $B\in V(W^{k,p}(\R))$ satisfying property \eqref{property}. To verify \eqref{condition_non_linearity} let $\|(V_1,W_1)\|_{W^{k,p}\times W^{k,p}} \leq M$, $\|(V_2,W_2)\|_{W^{k,p}\times W^{k,p}}\leq M$. Then, we can bound as follows
\begin{align*}
R\! &=\|A(V_1,W_1) B(V_1,W_1)-  A(V_2,W_2)B(V_2,W_2)\|_{W^{k,p}\times W^{k,p}}\\
&= \|A(V_1,W_1) B(V_1,W_1)- A(V_2,W_2)B(V_1,W_1) \\
&\hspace{0.3cm}+ A(V_2,W_2)B(V_1,W_1)- A(V_2,W_2)B(V_2,W_2)\|_{W^{k,p}\times W^{k,p}} \\
&\leq \|A(V_1,W_1) B(V_1,W_1)- A(V_2,W_2)B(V_1,W_1)\|_{W^{k,p}\times W^{k,p}}\\
& \hspace{0.3cm}+ \|A(V_2,W_2)B(V_1,W_1)- A(V_2,W_2)B(V_2,W_2)\|_{W^{k,p}\times W^{k,p}}\\
&\leq C\|A(V_1,W_1)\!-\!A(V_2,W_2)\|_{W^{k,p}\times W^{k,p}}\|B(V_1,W_1)\!-\!B(0,0)\|_{W^{k,p}\times W^{k,p}}\\
&\hspace{0.1cm}+\! C\|A(V_1,W_1)\!-\!A(V_2,W_2)\|_{W^{k,p}\times W^{k,p}}\|B(0,0)\|_{W^{k,p}\times W^{k,p}}\\
&\hspace{0.1cm}+\! C\|A(V_2,W_2)\!-\!A(0,0)\|_{W^{k,p}\times W^{k,p}}\|B(V_1,W_1)\!-\!B(V_2,W_2)\|_{W^{k,p}\times W^{k,p}}\\
&\hspace{0.1cm}+\!C\|A(0,0)\|_{W^{k,p}\times W^{k,p}}\|B(V_1,W_1) \!-\! B(V_2,W_2)\|_{W^{k,p}\times W^{k,p}}\\
&\leq C \|(V_1-V_2,W_1-W_2)\|_{W^{k,p}\times W^{k,p}}\\
&\hspace{0.3cm}\times\Big( \|(V_1,W_1)\|_{W^{k,p}\times W^{k,p}}+\|(V_2,W_2)\|_{W^{k,p}\times W^{k,p}} \\
&\hspace{1cm}+ \|A(0,0)\|_{W^{k,p}\times W^{k,p}}+\|B(0,0)\|_{W^{k,p}\times W^{k,p}}\Big)
\end{align*}
where we used Lemma \ref{interpolation_lemma} item (ii). This gives us \eqref{condition_non_linearity}. Condition \eqref{condition_non_linearity3} is analogous. Condition \eqref{condition_non_linearity2} is more delicate since it involves two different spaces. Let $\|(V_1,W_1)\|_{L_x^\infty \times L_x^\infty}\leq M$ and apply Lemma \ref{interpolation_lemma} items (ii) and (iii) to obtain
\begin{align*}
&R =\|A(V_1,W_1) B(V_1,W_1)-  A(0,0)B(0,0)\|_{W^{k,p}\times W^{k,p}}\\
&= \|A(V_1,W_1) B(V_1,W_1)- A(V_1,W_1)B(0,0) \\
&\hspace{0.3cm}+ A(V_1,W_1)B(0,0)- A(0,0)B(0,0)\|_{W^{k,p}\times W^{k,p}} \\
&\leq \|(A(V_1,W_1)-A(0,0)) (B(V_1,W_1)- B(0,0))\|_{W^{k,p}\times W^{k,p}} \\
&\hspace{0.3cm}+ \|A(0,0) (B(V_1,W_1)- B(0,0))\|_{W^{k,p}\times W^{k,p}}\\
&\hspace{0.3cm}+ \|(A(V_1,W_1)\!-\!A(0,0))B(0,0)\|_{W^{k,p}\times W^{k,p}}\\
&\leq\! C(\|A(V_1,W_1)\!-\!A(0,0)\|_{W^{k,p}\times W^{k,p}}\!+\!\|B(V_1,W_1)-B(0,0)\|_{W^{k,p}\times W^{k,p}})\\
&\hspace{0.6cm}\times (\|A(V_1,W_1)-A(0,0)\|_{L_x^\infty\times L_x^\infty}+\|B(V_1,W_1)-B(0,0)\|_{L_x^\infty \times L_x^\infty})\\
&\hspace{0.3cm}+C\|A(0,0)\|_{W^{k,p}}\|B(V_1,W_1)-B(0,0)\|_{W^{k,p}\times W^{k,p}}\\
&\hspace{0.3cm}+ C\|A(V_1,W_1)-A(0,0)\|_{W^{k,p}\times W^{k,p}}\|B(0,0)\|_{W^{k,p}\times W^{k,p}}\\
&\leq\! C \|(V_1,W_1)\|_{W^{k,p}\times W^{k,p}}\\
&\hspace{0.3cm}\times\left( \|(V_1,W_1)\|_{L_x^\infty \times L_x^\infty} + \|A(0,0)\|_{W^{k,p}\times W^{k,p}}+\|B(0,0)\|_{W^{k,p}\times W^{k,p}}\right).
\end{align*}
We conclude that $AB$ also satisfies property \eqref{property}.
\item[c)] Trivial.
\end{enumerate}
By structural induction we conclude that the conditions \eqref{condition_non_linearity}, \eqref{condition_non_linearity2}, \eqref{condition_non_linearity3} are satisfied for every element in $V(W^{k,p}(\R))$. This concludes the proof of Proposition \ref{prop_hip_local_existence_PAS} for $\mathcal{B} = W^{k,p}(\R)$. 
\bigbreak  
\noindent
{\bf Case $\mathcal{B} = BC^0(\R)\cap L^p(\R)$}\\
It is enough to notice that if  $\|f\|_\mathcal{B} = \|f\|_{L^\infty} + \|f\|_{L^p}$ then
\begin{equation*}
\|f g\|_{\mathcal{B}}\leq \|f\|_{L^\infty} \|g\|_{L^\infty}+ \|f\|_{L^\infty} \|g\|_{L^p} + \|f\|_{L^p} \|g\|_{L^\infty}.
\end{equation*}
Two immediate consequences of this inequality are
$$\|f g\|_\mathcal{B}\leq 2 \|f\|_\mathcal{B} \|g\|_\mathcal{B},$$
and 
$$\|f g\|_\mathcal{B}\leq \|f\|_{L^\infty} \|g\|_\mathcal{B} + \|f\|_\mathcal{B} \|g\|_{L^\infty}.$$
Thanks to these two inequalities, the previous proof also works in this case.  This concludes the proof of Proposition \ref{prop_hip_local_existence_PAS} for $\mathcal{B} = BC^0(\R)\cap L^{p}(\R)$.

\qed


\end{document}